\documentclass{amsart}
\usepackage[margin=2.8cm]{geometry}
\usepackage{amssymb}
\usepackage{graphicx} 
\usepackage[mathscr]{euscript}
\usepackage{enumerate}
\usepackage{xspace}
\usepackage{color}
\usepackage{enumerate}
\usepackage{enumitem}
\usepackage{amsthm}
\usepackage{dsfont}
\usepackage{mathrsfs}
\usepackage{bbm}
\usepackage[colorlinks=true, linkcolor = blue, citecolor = blue]{hyperref}
\usepackage[numbers]{natbib}
\usepackage[backgroundcolor=white,bordercolor=red]{todonotes}
\DeclareMathAlphabet{\mathpzc}{OT1}{pzc}{m}{it}
\usepackage[nameinlink]{cleveref}
\numberwithin{equation}{section}
\begin{document}

\theoremstyle{plain}

\newtheorem{theorem}{Theorem}[section]
\newtheorem{lemma}[theorem]{Lemma}
\newtheorem{example}[theorem]{Example}
\newtheorem{proposition}[theorem]{Proposition}
\newtheorem{corollary}[theorem]{Corollary}
\newtheorem{definition}[theorem]{Definition}
\newtheorem{Ass}[theorem]{Assumption}
\newtheorem{condition}[theorem]{Condition}
\theoremstyle{definition}
\newtheorem{remark}[theorem]{Remark}
\newtheorem{SA}[theorem]{Standing Assumption}

\newcommand{\of}{[\hspace{-0.06cm}[}
\newcommand{\gs}{]\hspace{-0.06cm}]}

\newcommand\llambda{{\mathchoice
		{\lambda\mkern-4.5mu{\raisebox{.4ex}{\scriptsize$\backslash$}}}
		{\lambda\mkern-4.83mu{\raisebox{.4ex}{\scriptsize$\backslash$}}}
		{\lambda\mkern-4.5mu{\raisebox{.2ex}{\footnotesize$\scriptscriptstyle\backslash$}}}
		{\lambda\mkern-5.0mu{\raisebox{.2ex}{\tiny$\scriptscriptstyle\backslash$}}}}}

\newcommand{\1}{\mathds{1}}

\newcommand{\F}{\mathbf{F}}
\newcommand{\G}{\mathbf{G}}

\newcommand{\B}{\mathbf{B}}

\newcommand{\M}{\mathcal{M}}

\newcommand{\la}{\langle}
\newcommand{\ra}{\rangle}

\newcommand{\lle}{\langle\hspace{-0.085cm}\langle}
\newcommand{\rre}{\rangle\hspace{-0.085cm}\rangle}
\newcommand{\blle}{\Big\langle\hspace{-0.155cm}\Big\langle}
\newcommand{\brre}{\Big\rangle\hspace{-0.155cm}\Big\rangle}

\newcommand{\X}{\mathsf{X}}

\newcommand{\bx}{\mathsf{x}}
\newcommand{\bX}{\mathsf{X}}

\newcommand{\tr}{\operatorname{tr}}
\newcommand{\N}{{\mathbb{N}}}
\newcommand{\cadlag}{c\`adl\`ag }
\newcommand{\on}{\operatorname}
\newcommand{\oP}{\overline{P}}
\newcommand{\oQ}{\overline{Q}}
\newcommand{\oO}{\mathcal{O}}
\newcommand{\D}{D(\mathbb{R}_+; \mathbb{R})}

\renewcommand{\epsilon}{\varepsilon}

\newcommand{\fPs}{\mathfrak{P}_{\textup{sem}}}
\newcommand{\fPas}{\mathfrak{P}^{\textup{ac}}_{\textup{sem}}}
\newcommand{\rrarrow}{\twoheadrightarrow}
\newcommand{\cC}{\mathcal{C}}
\newcommand{\cK}{\mathcal{K}}
\newcommand{\cH}{\mathcal{H}}
\newcommand{\cD}{\mathcal{D}}
\newcommand{\cE}{\mathcal{E}}
\newcommand{\cR}{\mathcal{R}}
\newcommand{\cQ}{\mathcal{Q}}
\newcommand{\cF}{\mathcal{F}}
\newcommand{\bth}{\overset{\leftarrow}\theta}
\renewcommand{\th}{\theta}

\newcommand{\bR}{\mathbb{R}}
\newcommand{\nnabla}{\nabla}
\newcommand{\f}{\mathfrak{f}}
\newcommand{\g}{\mathfrak{g}}
\newcommand{\oconv}{\overline{\operatorname{conv}}\hspace{0.1cm}}
\newcommand{\usa}{\on{usa}}
\newcommand{\usc}{\textit{USC}}
\newcommand{\uc}{\textit{UC}}
\newcommand{\lip}{\textit{Lip}}
\newcommand{\C}{\mathsf{C}}
\newcommand{\ou}{\overline{u}}
\newcommand{\ua}{\underline{a}}
\newcommand{\uu}{\underline{u}}
\newcommand{\p}{\mathsf{P}}
\newcommand{\cU}{\mathcal{U}}

\renewcommand{\emptyset}{\varnothing}

\allowdisplaybreaks

\makeatletter
\@namedef{subjclassname@2020}{%
	\textup{2020} Mathematics Subject Classification}
\makeatother

 \title[Markov Selections and Feller Properties of nonlinear Diffusions]{Markov Selections and Feller Properties \\ of nonlinear Diffusions}
\author[D. Criens]{David Criens}
\author[L. Niemann]{Lars Niemann}
\address{Albert-Ludwigs University of Freiburg, Ernst-Zermelo-Str. 1, 79104 Freiburg, Germany}
\email{david.criens@stochastik.uni-freiburg.de}
\email{lars.niemann@stochastik.uni-freiburg.de}



\thanks{We thank two anonymous referees for many comments and suggestions that helped us to improve the paper.}

\thanks{DC acknowledges financial support from the DFG project SCHM 2160/15-1 and LN acknowledges financial support from the DFG project SCHM 2160/13-1.}
\date{\today}

\maketitle

\begin{abstract}
In this paper we study a family of nonlinear (conditional) expectations that can be understood as a diffusion with uncertain local characteristics. Here, the differential characteristics are prescribed by a set-valued function.
We establish its Feller properties and examine how to linearize the associated sublinear Markovian semigroup. 
In particular, we observe a novel smoothing effect of sublinear semigroups in frameworks which carry enough randomness.
Furthermore, we link the value function corresponding to the semigroup to a nonlinear Kolmogorov equation. This provides a connection to the so-called Nisio semigroup.
\end{abstract}

\section{Introduction}
A \emph{nonlinear diffusion}, or \emph{nonlinear continuous Markov process}, is a family of sublinear expectations \( \{\cE^x \colon x \in \bR\} \) on the Wiener space \( C(\bR_+; \bR) \) with
\( \cE^x \circ X_0^{-1} = \delta_x \) for each \( x \in \bR \) such that the Markov property
\begin{equation} \label{eq: markov property}
\cE^x(\cE^{X_t}(\psi (X_s))) = \cE^x(\psi (X_{t+s}) ), \quad x \in \bR, \ s,t \in \bR_+, 
\end{equation}
holds. Here, \(\psi\) runs through a collection of suitable test functions and \( X \) denotes the canonical process on \( C(\bR_+; \bR) \).
Building upon the pioneering work of Peng \cite{peng2007g, peng2008multi} on the \(G\)-Brownian motion, nonlinear Markov processes have been intensively studied in recent years, both from the perspective of processes under uncertainty  \cite{fadina2019affine, hu2021g, neufeld2017nonlinear}, as well as sublinear Markovian semigroups \cite{denk2020semigroup,GNR22, hol16,K21, NR}.
Using the techniques developed in \cite{NVH}, a general framework for constructing nonlinear Markov processes was established in \cite{hol16}.
To be more precise, for given \( x \in \bR\), the sublinear expectation
\( \cE^x \) has the form
\( \cE^x = \sup_{P \in \cR(x)} E^P \)
with a collection \( \cR(x) \) of semimartingale laws \(P\) on the path space, with initial distribution \( \delta_x \) and absolutely continuous semimartingale characteristics \((B^{P}, C^{P})\), where the differential characteristics \((dB^{P} /d\llambda, dC^{P}/d\llambda)\) are prescribed in a Markovian way. In this paper, we parameterize drift and quadratic variation by a compact parameter space \(F\) and two functions \(b \colon F \times \bR \to \bR\) and \(a \colon F \times \bR \to \bR_+\) such that
\[
\cR (x) := \big\{ P \in \fPas \colon P \circ X_0^{-1} = \delta_{x},\ (\llambda \otimes P)\text{-a.e. } (dB^{P} /d\llambda, dC^{P}/d\llambda) \in \Theta(X)   \big\}, \quad x \in \bR,
\]
where
\[
\Theta (x) := \big\{(b (f, x), a (f, x)) \colon f \in F \big\}, \quad x \in \bR.
\]

As in the theory of (linear) Markov processes, there is a strong link to semigroups. Indeed, the Markov property \eqref{eq: markov property} ensures the semigroup property \( T_t T_s = T_{s+t}, s,t \in \bR_+ \), where the sublinear operators \( T_t, t \in \bR_+, \) are defined by 
\begin{equation} \label{eq: def semigroup}
    T_t(\psi)(x) := \cE^x(\psi(X_t)) = \sup_{P \in \cR(x)} E^P \big[ \psi(X_t) \big]
\end{equation}
for suitable functions \(\psi\).
Using the general theory of \cite{ElKa15, NVH}, the operators
\(T_t, t \in \mathbb{R}_+\), are well-defined on the cone of upper semianalytic functions. 

The purpose of this article is to study two aspects of nonlinear diffusions.
First, we examine the \emph{Feller property} of its associated semigroup and second, we investigate how to \emph{linearize} a nonlinear diffusion, respectively its associated semigroup.
Let us explain our contributions in more detail.

In the case of nonlinear  L\'evy processes, i.e., where the set of possible local characteristics is independent of time and path, Equation \eqref{eq: def semigroup} takes the form
\begin{equation} \label{eq: levy semigroup}
    T_t(\psi)(x) = \sup_{P \in \cR(0)} E^P \big[ \psi(x + X_t) \big] = \cE^0(\psi(x + X_t)),
\end{equation}
and the additive structure in \eqref{eq: levy semigroup} gives access to the \emph{\(C_b\)--Feller property} of \( (T_t)_{t \in \bR_+} \), i.e., \( T_t(C_b(\bR; \bR)) \) \(\subset C_b(\bR; \bR)\) for all \(t \in \mathbb{R}_+\).
However, in general, this property seems to be hard to verify, see \cite[Remark~4.43]{hol16}, \cite[Remark 3.4]{K19} and \cite[Remark 5.4]{K21} for comments. 

By virtue of Berge's maximum theorem, it is a natural idea to deduce regularity properties of the sublinear semigroup \((T_t)_{t \in \mathbb{R}_+}\) from the corresponding properties of the set-valued mapping \(x \mapsto \cR(x)\). The strategy to deduce certain regularity properties from related properties of a correspondence is not new. For instance, it has been used in the seminal paper \cite{nicole1987compactification} to obtain conditions for upper and lower semicontinuity of a value function in a relaxed framework for controlled diffusions.
In our setting, we prove that \(x \mapsto \cR(x)\) is upper hemicontinuous and compact-valued in case \(b\) and \(a\) are continuous and of linear growth, and the set \(\Theta(x)\) is convex for every \(x \in \bR\). This result establishes the \emph{\(\usc_b\)--Feller property} of \( (T_t)_{t \in \bR_+} \), i.e., \( T_t( \usc_b(\bR; \bR)) \subset \usc_b(\bR; \bR) \) for all~\(t \in \mathbb{R}_+\). 

Lower hemicontinuity of the correspondence \(x \mapsto \cR(x)\) appears to be more difficult to establish.
In particular, the example from \cite{SV} for the non-existence of a Feller selection from the set of solutions to a non-wellposed martingale problem shows that our conditions for the upper hemicontinuity of \(x \mapsto \mathcal{R} (x)\) are not sufficient for its lower hemicontinuity.
There are indications in the literature that additional Lipschitz conditions on \(b\) and \(a\) might suffice for lower hemicontinuity. Indeed, such a result was established in \cite{nicole1987compactification} for the relaxed framework of controlled diffusions. After this paper was submitted, local Lipschitz conditions for lower hemicontinuity of a fully path-dependent nonlinear continuous semimartingale framework have been proved in an update of our paper~\cite{CN22}.

In this paper we present a new approach to show Feller properties of \( (T_t)_{t \in \bR_+} \)
by means of a \emph{Feller selection principle}.
Let us detail our ideas.
As the correspondence \( x \mapsto \cR (x) \) is compact-valued, for every upper semicontinuous function \( \psi \colon C(\bR_+; \bR) \to \bR \), and every \( x \in \bR \), there exists a measure \(P_x \) in the set of maximizers \( \cR^*(x) \) of~\eqref{eq: def semigroup}. Building on ideas of Krylov \cite{krylov1973selection} about Markovian selection, and leaning on the techniques from \cite{nicole1987compactification, hausmann86, SV}, we show that for every bounded and upper semicontinuous \( \psi \colon \bR \to \bR\) and every \(t > 0\), there exists a time inhomogeneous \emph{strong Markov selection} \(\{P_{(s, x)} \colon (s, x) \in \bR_+ \times \bR\}\) such that~\(P_{(0, x)} \in \cR (x)\) and
\begin{equation*}
T_t (\psi) (x) = E^{P_{(0, x)}} \big[ \psi (X_t) \big] .
\end{equation*}
Under the additional ellipticity assumption \(a > 0\), using some fundamental results of Stroock and Varadhan \cite{SV}, we prove that the strong Markov family \(\{P_{(t, x)} \colon (t, x) \in \bR_+ \times \bR\}\) is even a (time inhomogeneous) strong Feller family. In particular, this Feller selection principle shows that the semigroup \( (T_t)_{t \in \bR_+} \) has the \emph{strong \(\usc_b\)--Feller property}, i.e.,  \(T_t (\usc_b (\mathbb{R}; \mathbb{R})) \subset C_b (\mathbb{R}; \mathbb{R})\) for all \(t > 0\). Under a uniform ellipticity and boundedness assumption, it even follows that \((T_t)_{t \in \bR_+}\) has the {\em uniform strong Feller property} in the sense that \(T_t\) for \(t > 0\) maps bounded upper semicontinuous functions to bounded uniformly continuous functions. 
To the best of our knowledge, \emph{smoothing effects} of this specific form were not reported before in the context of nonlinear Markov processes. Related, but generally different, smoothing effects are known for viscosity solutions to parabolic Hamilton--Jacobi--Bellman (HJB) (or more general nonlinear) PDEs, see, e.g., \cite{Cra02,kry18, kry17} and the references therein. Under suitable conditions, in the HJB case a continuous terminal function leads to viscosity solutions of class \(C^{1 + \alpha}_{\on{loc}}\) (cf. \cite{Cra02} for details on this result and the notation). In our situation, plugging a possibly {\em discontinuous} function into a sublinear semigroup leads to a {\em continuous} function. 

We emphasise that our idea of first proving the existence of a strong Markov selection and then verifying its (strong) Feller property extends to higher dimensional situations. 

Let us also relate our result to the recent local Lipschitz conditions from \cite{CN22}. It is well-known in the literature on stochastic differential equations (SDEs) that (local) Lipschitz conditions imply pathwise uniqueness (even in the presence of random coefficients, see, e.g., \cite{jacod79}). Indeed, the proof from \cite{CN22} for lower hemicontinuity relies in a crucial manner on the strong existence and pathwise uniqueness of SDEs with random coefficients. 
We are not aware of ellipticity conditions for existence and uniqueness of SDEs with random coefficients. Therefore, we think that a new strategy as proposed in this paper is useful.

Sublinear semigroups can also be constructed by analytic methods. A general approach leading to the so-called {\em Nisio semigroup} and a corresponding viscosity theory was recently established in the paper \cite{NR}. The framework from \cite{NR} allows for general state spaces, and provides conditions for Feller properties on spaces of weighted continuous functions. If the weight functions is vanishing at infinity, this includes the \(C_b\)--Feller property. Further, also in case the weight function is bounded from below and the Nisio semigroup is continuous from above (in a suitable sense), the \(C_b\)--Feller property can be derived.
For the case of convolution semigroups, corresponding to the L\'evy framework, it has been shown in \cite{K21} that \((T_t)_{t \in \bR_+}\) coincides with the Nisio semigroup. 

This relation is based on the link between the semigroup and its generator. We investigate this for our framework. More precisely,
the \(C_b\)--Feller property allows us to identify the so-called value function \([0, T] \times \mathbb{R} \ni (t, x) \mapsto v (t, x) := \cE^x(\psi(X_{T-t})) \) as a bounded viscosity solution to the nonlinear Kolmogorov type PDE
\begin{equation} \label{eq: intro PDE}
\begin{cases}   
\partial_t v (t, x) + G (t, x, v) = 0, & \text{for } (t, x) \in [0, T) \times \mathbb{R}, \\
v (T, x) = \psi (x), & \text{for } x \in \bR,
\end{cases}
\end{equation}
where
\begin{align*}
    G(t, x, \phi) := \sup \Big\{b (f, x) \partial_x \phi (t,x)
    + \tfrac{1}{2} a (f, x) \partial^2_x\phi (t,x) \colon f \in F \Big\}.
\end{align*}
Under suitable Lipschitz and boundedness conditions, we can use uniqueness results for the generator equation \eqref{eq: intro PDE} to show that the semigroup \((T_t)_{t \in \bR_+}\) from \eqref{eq: def semigroup} coincides with the Nisio semigroup from \cite{NR} on the space \(\uc_b (\bR; \bR)\) of bounded uniformly continuous functions from \(\bR\) into \(\bR\). In particular, this relation shows that \((T_t)_{t\in \bR_+}\) is a sublinear semigroup on \(\uc_b(\bR; \bR)\) under these conditions. 

As a referee has pointed out, under Lipschitz conditions, one can prove that the upper and lower envelopes of the value function are viscosity sub- and supersolutions, respectively. A (strong) comparison result (see, e.g., \cite{Pham}) then implies that the value function is already the unique bounded viscosity solution. In particular, its continuity is established en passant. We also detail this (mainly) analytic approach in this paper.

\smallskip
Let us now explain the idea of \emph{linearization} of a sublinear Markovian semigroup.
In the presence of uncertainty, it is not possible to choose a single family \( \{ P_x \colon x \in \bR \} \) such that
\begin{equation*}
T_t (\psi) (x) = E^{P_x} \big[ \psi (X_t) \big]
\end{equation*}
for all functions \( \psi \) and \(t > 0\).
However, under some structural assumptions, for example in the case without drift, we are able to construct a time homogeneous \emph{strong Feller family} \( \{ P^*_x \colon x \in \bR \} \) such that
\begin{equation} \label{eq: intro convex}
\cE^x(\psi(X_t)) = E^{P^*_x} \big[ \psi(X_t) \big] ,
\end{equation}
for \emph{all} convex functions \( \psi \colon \bR \to \bR \) of polynomial growth, and all \( t \in \bR_+ \). We derive \eqref{eq: intro convex} by means of convex stochastic ordering, where we adapt techniques from \cite{hajek85}. Similarly, we present a linearization result for the class of increasing Borel functions \(\psi \colon C(\mathbb{R}_+;\mathbb{R}) \to \mathbb{R}\) in case of certain volatility.
Finally, this linearization allows us to simplify the PDE \eqref{eq: intro PDE}. More precisely, we prove that for convex \(\psi \) of polynomial growth, the  function  \( (t, x) \mapsto \cE^x(\psi(X_{T - t})) \) is the unique viscosity solution (of polynomial growth) to the \emph{linear}~PDE
\begin{equation} \label{eq: intro linear PDE}
\begin{cases}   
\partial_t u (t, x) + \frac{1}{2}a^*(x) \partial_x^2 u(t,x) = 0, & \text{for } (t, x) \in [0, T) \times \mathbb{R}, \\
u (T, x) = \psi (x), & \text{for } x \in \bR,
\end{cases}
\end{equation}
where \(a^*(x) := \sup\{ a(f,x) \colon f \in F \} \).
Additionally, this linearization allows us to deduce that
\( (t, x) \mapsto \cE^x(\psi(X_{T - t})) \) is the unique \emph{classical solution} of \eqref{eq: intro linear PDE}.

\smallskip
This paper is structured as follows: in Section \ref{subsec: setting} we introduce our setting. Section \ref{subsec: markovian semigroups} is devoted to the construction of sublinear Markovian semigroups and their Feller properties. Section \ref{subsec: linearization} shows how to linearize a sublinear Markovian semigroup, while Section \ref{subsec: viscosity} links the nonlinear expectation to the nonlinear Kolmogorov equation \eqref{eq: intro PDE}. Section \ref{sec: nisio} investigates the link to the Nisio semigroup. Section \ref{sec: regularity r} establishes the required regularity of the set-valued mapping \( \cR \). 

The proofs for our main results are given in the remaining sections. 
More precisely, the Feller properties of sublinear Markovian semigroups are shown in Section \ref{sec: proof main result}, 
the selection principles are proved in Section~\ref{sec: feller selection}, the linearization is proved in Section~\ref{sec: uniform selection}
and the PDE connection is proved in Section~\ref{sec: viscosity property}. In Section~\ref{sec: pf nisio} the proofs for the relation to the Nisio semigroup are given.

\section{Main Results}
\subsection{The Setting}\label{subsec: setting}
Define $\Omega$ to be the space of continuous functions \(\mathbb{R}_+ \to \mathbb{R}\) endowed with the local uniform topology. 
The canonical process on $\Omega$ is denoted by \(X\), i.e., \(X_t (\omega) = \omega (t)\) for \(\omega \in \Omega\) and \(t \in \mathbb{R}_+\). 
It is well-known that \(\mathcal{F} := \mathcal{B}(\Omega) = \sigma (X_t, t \in \bR_+)\).
We define $\F := (\mathcal{F}_t)_{t \in \bR_+}$ as the canonical filtration generated by $X$, i.e., \(\mathcal{F}_t := \sigma (X_s, s \in [0, t])\) for \(t \in \mathbb{R}_+\). Notice that we do not make the filtration \(\F\) right-continuous. The set of probability measures on \((\Omega, \mathcal{F})\) is denoted by \(\mathfrak{P}(\Omega)\) and endowed with the usual topology of convergence in distribution.
Let \(F\) be a Polish space and 
let \(b \colon F \times \mathbb{R} \to \mathbb{R}\) and \(a \colon F \times \mathbb{R} \to \mathbb{R}_+\) be Borel functions. We define the correspondence, i.e., the set-valued mapping, \(\Theta \colon \bR \twoheadrightarrow \mathbb{R} \times \mathbb{R}_+\) by
\[
\Theta (x) := \big\{(b (f, x), a (f, x)) \colon f \in F \big\}.
\]

\begin{SA} \label{SA: compact}
\(F\) is compact.
\end{SA}

\begin{SA} \label{SA: meas gr}
\(\Theta\) has a measurable graph, i.e., the graph
\[
\on{gr} \Theta = \big\{ (x, b, a) \in \bR \times \mathbb{R} \times \mathbb{R}_+ \colon (b, a) \in \Theta (x) \big\}
\]
is Borel. 
\end{SA}
\begin{remark}
        By virtue of \cite[Lemma 2.8]{CN22}, if Standing Assumption \ref{SA: compact} is in force,  Standing Assumption~\ref{SA: meas gr} holds once \(b\) and \(a\) are continuous in their first variables.
\end{remark}

We call a real-valued continuous process \(Y = (Y_t)_{t \geq 0}\) a (continuous) \emph{semimartingale after a time \(t^* \in \mathbb{R}_+\)} if the process \(Y_{\cdot + t^*} = (Y_{t + t^*})_{t \geq 0}\) is a semimartingale for its natural right-continuous filtration.
Notice that it comes without loss of generality that we consider the right-continuous version of the filtration (see \cite[Proposition~2.2]{neufeld2014measurability}). 
The law of a semimartingale after \(t^*\) is said to be a \emph{semimartingale law after \(t^*\)} and the set of them is denoted by \(\fPs (t^*)\). 
Notice also that \(P \in \fPs(t^*)\) if and only if the coordinate process is a semimartingale after \(t^*\).
For \(P \in \fPs (t^*)\) we denote the semimartingale characteristics of the shifted coordinate process \(X_{\cdot + t^*}\) by \((B^P_{\cdot + t^*}, C^P_{\cdot + t^*})\). 
Moreover, we set 
\[
\fPas (t^*) := \big\{ P \in \fPs (t^*) \colon P\text{-a.s. } (B^P_{\cdot + t^*}, C^P_{\cdot + t^*}) \ll \llambda \big\}, \quad \fPas := \fPas(0),
\]
where \(\llambda\) denotes the Lebesgue measure.
For \( x \in \bR \), we define
\[
\cR (x) := \big\{ P \in \fPas \colon P \circ X_0^{-1} = \delta_{x},\ (\llambda \otimes P)\text{-a.e. } (dB^{P} /d\llambda, dC^{P}/d\llambda) \in \Theta(X)   \big\}.
\]
While the correspondence \( \cR \) is in the focus of interest for our study, it is convenient to introduce another correspondence \( \cC \).
For \((t,\omega) \in \of 0, \infty\of \,:= \bR_+ \times \Omega\), we define 
\begin{align*}
\cC(t,\omega) := \big\{ P \in \mathfrak{P}_{\text{sem}}^{\text{ac}}(t)\colon P(&X = \omega \text{ on } [0, t]) = 1, \\ &(\llambda \otimes P)\text{-a.e. } (dB^P_{\cdot + t} /d\llambda, dC^P_{\cdot + t}/d\llambda) \in \Theta (X_{\cdot + t}) \big\}.
\end{align*}
Note that \( \cC(0,x) = \cR(x) \) for every \( x \in \bR\).

\begin{SA} \label{SA: non empty}
\(\cC (t, \omega) \not = \emptyset\) for all \((t, \omega) \in \of 0, \infty\of\).
\end{SA}
\begin{remark} By virtue of \cite[Lemma 2.10]{CN22}, Standing Assumption \ref{SA: non empty} holds under continuity and linear growth conditions on \(b\) and \(a\). In particular, Standing Assumption \ref{SA: non empty} is implied by the Conditions \ref{cond: LG} and \ref{cond: continuity} below.
\end{remark}


\subsection{Sublinear Markovian Semigroups} \label{subsec: markovian semigroups}

For each \( x \in \bR \), we define the sublinear operator \( \cE^x \) on the convex cone of upper semianalytic functions \(\psi \colon \Omega \to \bR\) by
\( \cE^x(\psi) := \sup_{P \in \cR(x)} E^P[ \psi ] \). For every \( x \in \bR \), we have by construction that \( \cE^x(\psi(X_0)) = \psi(x) \) for every upper semianalytic function \(\psi \colon \bR \to \bR\).

\begin{definition}
Let \( \mathcal{H} \) be a convex cone of functions \( f \colon \bR \to \bR \) containing all constant functions.
A family of sublinear operators \( T_t \colon \mathcal{H} \to \mathcal{H}, \ t \in \bR_+,\) is called a \emph{sublinear Markovian semigroup} on \( \mathcal{H} \) if it satisfies the following properties:
\begin{enumerate}
    \item[\textup{(i)}] \( (T_t)_{t \in \bR_+} \) has the semigroup property, i.e., 
          \( T_s T_t = T_{s+t} \) for all \(s, t \in \bR_+ \) and
          \( T_0 = \on{id} \),
    
    \item[\textup{(ii)}] \( T_t \) is monotone for each \( t \in \bR_+\), i.e., 
    \( f, g \in \mathcal{H} \) with \( f \leq g \) implies \(T_t (f) \leq T_t (g) \),
    
    \item[\textup{(iii)}] \( T_t \) preserves constants for each  \( t \in \bR_+\), i.e.,
    \( T_t(c) = c \) for each \( c \in \bR  \).
\end{enumerate}
\end{definition}

The following proposition should be compared to \cite[Lemma 4.32]{hol16}. Note that the framework of \cite{ElKa15}, which is also used in our article, allows for more flexibility regarding initial values in \( \cR \) compared to \cite{hol16}.
Indeed, as it relies on the results of  \cite{NVH}, \cite{hol16} needs to introduce, for every \( x \in \bR\), the space \( \Omega_x := \{ \omega \in \Omega \colon w (0) =x\} \) to capture the initial value. In consequence, the sublinear expectation \( \cE^x \) constructed in \cite{hol16} is only defined on \( \Omega_x \), which requires more notational care. \vspace{1em}

Denote, for \(t \in \mathbb{R}_+\), the shift operator \(\theta_t \colon \Omega \to \Omega\) by \(\theta_t (\omega) := \omega(\cdot + t)\) for all \(\omega \in \Omega\).
\begin{proposition} \label{prop: markov property}
For every upper semianalytic function \( \psi \colon \Omega \to \bR \), the equality
\[
\cE^x( \psi \circ \theta_t) = \cE^x ( \cE^{X_t} (\psi))
\]
holds for every \((t, x) \in \bR_+ \times \bR\).
\end{proposition}
\begin{proof}
For every \( (t, \omega) \in \of 0, \infty \of \), we define 
\[
\mathcal{E}_t (\psi) (\omega) := \sup_{P \in \cC (t, \omega)} E^P \big[ \psi\big].
\]
Now, we get from
\cite[Corollary 7.3]{CN22} that
\[ \cE_t(\psi \circ \theta_t)(\omega) = \sup_{P \in \cR(\omega(t))} E^{P}\big[\psi( \omega(t) + X - X_0) \big] = \sup_{P \in \cR (\omega(t))} E^P \big[ \psi\big].
\]
Hence, 
\begin{align} \label{eq: equality markov shift} \cE_t(\psi \circ \theta_t)(\omega) = \cE^{\omega(t)}(\psi).
\end{align}
Finally, the dynamic programming principle (\cite[Theorem 3.1]{CN22}) yields
\[
    \cE^x(\psi \circ \theta_t )  = \cE^x( \cE_t( \psi \circ \theta_t)) 
    = \cE^x( \cE^{X_t}(\psi)).
\]
The proof is complete.
\end{proof}

We point out that Proposition \ref{prop: markov property} confirms the intuition that the coordinate process is a nonlinear Markov process under the family \( \{\cE^x \colon x \in \bR\} \), as it implies the equality
\[ \cE^x(\psi(X_{s+t})) = \cE^x( \cE^{X_t}( \psi (X_s))) \]
for every upper semianalytic function \(\psi \colon \bR \to \mathbb{R}\), \(s,t \in \bR_+ \) and \(x \in \bR\).
Using Proposition~\ref{prop: markov property}, the following proposition
is a restatement of \cite[Remark 4.33]{hol16} for our framework.

\begin{proposition}
The family of operators \( (T_t)_{t \in \bR_+} \) given by
\[ T_t ( \psi )(x) := \cE^x(\psi(X_t)), \quad (t, x) \in \bR_+ \times \bR,\]
defines a sublinear Markovian semigroup on the set of bounded upper semianalytic functions.
\end{proposition}

\begin{remark} \label{rem: nonMarkov laws}
It is worth pointing out that \(\cR(x)\) contains {\em non-}Markovian laws. This is already the case in the L\'evy situation where \(b (f, x) \equiv b (f)\) and \(a (f, x) \equiv a (f)\). To give a concrete example, consider \(F = [\underline{a}, \overline{a}]\) for \(\underline{a} < \overline{a}\), \(b = 0\) and \(a (f, x) = f\). Then, \(\{\cE^x \colon x \in \bR\}\) corresponds to a \(G\)-Brownian motion with \(G\)-function \(G (x) = (\overline{a} x^+ - \underline{a} x^-) / 2\). In this case, the set \(\cR (x)\) contains also laws of processes of the type
    \[
    d Y_t = \sigma (t, Y) \, d W_t, \quad Y_0 = x,
    \]
    where \(\sigma \colon \of 0, \infty\of \, \to \mathbb{R}\) is an arbitrary predictable functional that is continuous on \(\of 0, \infty\of\) and such that \(\underline{a} \leq \sigma^2 \leq \overline{a}\).\footnote{The continuity and boundedness assumptions on \(\sigma\) entail the existence of \(Y\) by Skorokhod's existence theorem.} Such processes are non-Markovian in general.  
\end{remark}

In the following, we investigate the semigroup property of \((T_t)_{t \in \mathbb{R}_+}\) on convex cones consisting of more regular functions.

\begin{definition}
We say that the sublinear Markovian semigroup \((T_t)_{t \in \mathbb{R}_+}\) has the 
\begin{enumerate}
    \item[\textup{(a)}]
    \emph{\(\usc_b\)--Feller property} if it is a sublinear Markovian semigroup on the space \(\usc_b(\mathbb{R}; \mathbb{R})\) of bounded upper semicontinuous functions from \(\mathbb{R}\) into \(\mathbb{R}\);
            \item[\textup{(b)}]
        \emph{\(C_b\)--Feller property} if it is a sublinear Markovian semigroup on the space \(C_b(\mathbb{R}; \mathbb{R})\) of bounded continuous functions from \(\mathbb{R}\) into \(\mathbb{R}\);
         \item[\textup{(c)}]
        \emph{\(\uc_b\)--Feller property} if it is a sublinear Markovian semigroup on the space \(\uc_b(\mathbb{R}; \mathbb{R})\) of bounded uniformly continuous\footnote{Of course, uniform continuity refers to the Euclidean metric.} functions from \(\mathbb{R}\) into \(\mathbb{R}\);
    \item[\textup{(d)}]
          \emph{\(C_0\)--Feller property} if it is a sublinear Markovian semigroup on the space \(C_0(\mathbb{R}; \mathbb{R})\) of continuous functions from \(\mathbb{R}\) into \(\mathbb{R}\) which are vanishing at infinity;
    \item[\textup{(e)}] \emph{strong \(\usc_b\)--Feller property} if \(T_t (\usc_b (\mathbb{R}; \mathbb{R})) \subset C_b (\mathbb{R}; \mathbb{R})\) for all \(t > 0\).
        \item[\textup{(f)}] \emph{uniform strong \(\usc_b\)--Feller property} if \(T_t (\usc_b (\mathbb{R}; \mathbb{R})) \subset \uc_b (\mathbb{R}; \mathbb{R})\) for all \(t > 0\).
\end{enumerate}
\end{definition}
\begin{remark} \label{rem: schilling}
As observed in \cite{schilling98}, in case of linear semigroups, the \(\usc_b\)--Feller property is equivalent to the \(C_b\)--Feller property. Indeed, this follows simply from the fact that \(\usc_b (\mathbb{R}; \mathbb{R}) = - \textit{LSC}_b (\mathbb{R}; \mathbb{R})\), where the latter denotes the space of bounded lower semicontinuous functions.
\end{remark}
To formulate our main results we need to introduce some conditions.
\begin{condition}[Convexity] \label{cond: convexity}
	For every \(x \in \mathbb{R}\), the set \(\Theta(x)\) is convex.
\end{condition}
\begin{condition}[Linear Growth] \label{cond: LG}
	There exists a constant \(\C > 0\) such that 
	\[
	|b (f, x)|^2 + |a (f, x)| \leq \C\, ( 1 + |x|^2 )
	\]
	for all \(f \in F\) and \(x \in \mathbb{R}\).
\end{condition}

\begin{condition}[Boundedness]\label{cond: bdd}
    \(\sup\, \{ | b (f, x) | + a (f, x) \colon (f, x) \in F \times \bR \} < \infty\). 
\end{condition}

\begin{condition}[Continuity] \label{cond: continuity}
	\(b\) and \(a\) are continuous.
\end{condition}

\begin{condition}[Local Lipschitz Continuity] \label{cond: loc Lip}
    For every \(N > 0\), there exists a constant \(\C = \C (N) > 0\) such that 
    \[
    |b (f, x) - b(f, y)| + | \sqrt{a} (f, x) - \sqrt{a} (f, y)| \leq \C \, |x - y|, 
    \]
    for all \(f \in F\) and \(x, y \in [- N, N]\). 
\end{condition}

\begin{condition}[Ellipticity] \label{cond: ellipticity}
\(a > 0\).
\end{condition}

\begin{condition}[Uniform Ellipticity] \label{cond: uniform ellipticity}
\(\inf\, \{ a (f, x) \colon (f, x) \in F \times \bR\} > 0\).
\end{condition}

\begin{theorem} \label{thm: new very main}
Suppose throughout that the Conditions \ref{cond: convexity}, \ref{cond: LG} and \ref{cond: continuity} hold.
\begin{enumerate}
   \item[\textup{(i)}] The sublinear Markovian semigroup \((T_t)_{t \in \mathbb{R}_+}\) has the \(\usc_b\)--Feller property.
   \item[\textup{(ii)}] Assume that Condition~\ref{cond: loc Lip} holds. Then, \((T_t)_{t \in \bR_+}\) has the \(C_b\)--Feller property.
   \item[\textup{(iii)}]  Assume that Condition~\ref{cond: ellipticity} holds. Then, \((T_t)_{t \in \mathbb{R}_+}\) has the \(C_0\) and the strong \(\usc_b\)--Feller properties. 
   \item[\textup{(iv)}] Assume that the Conditions~\ref{cond: bdd} and \ref{cond: uniform ellipticity} hold. Then, \((T_t)_{t \in \bR_+}\) has the uniform strong \(\usc_b\)--Feller property.
\end{enumerate}
\end{theorem}

\begin{remark} \label{rem: no uniqueness}
	Let us shortly discuss the relation of Theorem \ref{thm: new very main} to the setting where \(a (f, x) = a(x)\) and \(b(f, x) = b(x)\) or, equivalently, where \(F\) is a singleton. It is important to notice that even in this case the semigroup \((T_t)_{t \in \mathbb{R}_+}\) is \emph{not necessarily} linear. Indeed, the nonlinearity stems from the possibility that the martingale problem associated to the coefficients \(b\) and \(a\) might not be well-posed.
	Furthermore, the example in \cite[Exercise 12.4.2]{SV} shows that the Conditions \ref{cond: convexity}, \ref{cond: LG} and \ref{cond: continuity} are \emph{not} sufficient for the \(C_b\) nor for the \(C_0\)--Feller property of \((T_t)_{t \in \mathbb{R}_+}\). In particular, this means that neither the local Lipschitz nor the ellipticity assumption \(a > 0\) can be dropped in Theorem \ref{thm: new very main} (ii) and (iii).
\end{remark}

For a nonlinear framework \emph{with jumps}, the \(\usc_b\)--Feller property was proved in \cite[Proposition~4.36, Theorem 4.41, Lemma 4.42]{hol16} under uniform boundedness and global Lipschitz conditions.\footnote{In the presence of jumps, there appears to be a gap in the proof of \cite[Lemma~4.42]{hol16}, as the map \(\omega \mapsto \omega (t)\) is not upper semicontinuous on the Skorokhod space of \cadlag functions. Indeed, by linearity, upper semicontinuity would already imply continuity, which is not the case.} Theorem \ref{thm: new very main} shows that in our framework one can weaken these assumptions substantially. 

The \(C_b\)--Feller property is of fundamental importance for the relation of nonlinear processes and semigroups. In Section \ref{subsec: viscosity} it enables us to derive a new existence and uniqueness result and a stochastic representation for certain Kolmogorov type PDEs. To the best of our knowledge, the uniform strong and the strong \(\usc_b\)--Feller properties of nonlinear Markov processes have not been investigated before. 
We highlight that both provide a smoothing effect which seems to be new in the literature on nonlinear Markov processes.

Let us now discuss the proofs of Theorem \ref{thm: new very main}. We start with (iii) and (iv), i.e., the \(C_0\) and (uniform) strong \(\usc_b\)--Feller properties, which we prove simultaneously. The first main tool for our proof is the following \emph{strong Markov selection principle} which we believe to be of interest in its own. Before we can state our result we need more notation and terminology.
For a probability measure \(P\) on \((\Omega, \mathcal{F})\), a kernel \(\Omega \ni \omega \mapsto Q_\omega \in \mathfrak{P}(\Omega)\), and a finite stopping time \(\tau\), we define the pasting measure
\[
(P \otimes_\tau Q) (A) := \iint \1_A (\omega \otimes_{\tau(\omega)} \omega') Q_\omega (d \omega') P(d \omega)
\]
for all \(A \in \cF\), where 
	\[
	\omega \otimes_t \omega' :=  \omega \1_{[ 0, t)} + (\omega (t) + \omega' - \omega' (t))\1_{[t, \infty)}.
	\] 

\begin{definition}[Time inhomogeneous Markov Family]
	A family \(\{P_{(s, x)} \colon (s, x) \in \bR_+ \times \bR\} \subset \mathfrak{P}(\Omega)\) is said to be a \emph{strong Markov family} if \((t, x) \mapsto P_{(t, x)}\) is Borel and the strong Markov property holds, i.e., for every \((s, x) \in \bR_+ \times \bR\) and every finite stopping time \(\tau \geq s\), 
	\[
	P_{(s, x)} (\,\cdot\, | \cF_\tau) (\omega) = \omega \otimes_{\tau (\omega)} P_{(\tau (\omega), \omega (\tau (\omega)))}
	\]
for \(P_{(s, x)}\)-a.a. \(\omega \in \Omega\).
\end{definition}

Further, we introduce a correspondence \(\mathcal{K} \colon \bR_+ \times \bR  \twoheadrightarrow \mathfrak{P}(\Omega)\) by 
\[
\cK (t, x) := \cC (t, \bx),
\]
where \(\bx \in \Omega\) is the constant function \(\bx (s) = x\) for all \(s \in \bR_+\).

\begin{theorem}[Strong Markov Selection Principle]  \label{theo: strong Markov selection}
Suppose that the Conditions \ref{cond: convexity}, \ref{cond: LG} and \ref{cond: continuity} hold.
    For every \(\psi \in \usc_b(\bR; \mathbb{R})\) and every \(t > 0\), there exists a strong Markov family \(\{P_{(s, x)} \colon (s, x) \in \bR_+ \times \mathbb{R}\}\) such that, for all \((s, x)\in \bR_+ \times \mathbb{R}\), \(P_{(s, x)} \in \cK (s, x)\) and 
    \[
    E^{P_{ (s, x) }} \big[ \psi (X_t) \big] = \sup_{P \in \cK (s, x)} E^P \big[ \psi (X_t) \big].
    \]
    In particular, for every \(x \in \bR\), 
    \[
    T_t (\psi) (x) = E^{P_{(0, x)}} \big[ \psi (X_t) \big]. 
    \]
\end{theorem}

In general, the set \(\cK\) contains non-Markovian laws (see Remark~\ref{rem: nonMarkov laws}). 
In this regard, it is interesting that the supremum of \(P \mapsto E^P [ \psi (X_t) ]\) over \((s, x) \mapsto \cK (s, x)\) is attained at a strong Markov family. We emphasise that the strong Markov selection depends on the input function \(\psi\) and the time \(t > 0\).

At first glance, it appears that abstract selection theorems only provide measurability in the initial value, as results on continuous selection, like Michael's theorem \cite[Theorem 3.2]{michael}, seem not applicable.
Under Condition~\ref{cond: ellipticity}, the system carries enough randomness to conclude additional regularity properties. More precisely, if \(a\) is elliptic, we prove in Theorem \ref{theo: selection is Feller} below that every strong Markov selection is already a \(C_0\) and strong Feller selection in the sense explained now.

\begin{definition} [Feller Properties of time inhomogeneous Markov Families]
	\quad
	\begin{enumerate}
		\item[\textup{(i)}] We say that a strong Markov family \(\{P_{ (s, x)} \colon (s, x) \in \bR_+ \times \bR\}\) has the \emph{\(C_b\)--Feller property} if, for every \(t > 0\) and every \(\phi \in C_b (\bR; \bR)\), the map \([0, t) \times \bR \ni (s, x) \mapsto E^{P_{(s, x)}} [ \phi (X_t)]\) is continuous.
		\item[\textup{(ii)}] We say that a strong Markov family \(\{P_{ (s, x)} \colon (s, x) \in \bR_+ \times \bR\}\) has the \emph{\(C_0\)--Feller property} if, for every \(0 \leq s < t\) and every continuous \(\phi \colon \bR \to \bR\) which is vanishing at infinity, the map \(x \mapsto E^{P_{(s, x)}} [ \phi (X_t)]\) is continuous and vanishing at infinity.
		\item[\textup{(iii)}] We say that a strong Markov family \(\{P_{ (s, x)} \colon (s, x) \in \bR_+ \times \bR\}\) has the \emph{strong Feller property} if, for every \(t > 0\) and every bounded Borel function \(\phi \colon \bR \to \bR\), the map \([0, t) \times \bR \ni (s, x) \mapsto E^{P_{(s, x)}}[ \phi (X_t) ]\) is continuous. 
  		\item[\textup{(iv)}] We say that a strong Markov family \(\{P_{ (s, x)} \colon (s, x) \in \bR_+ \times \bR\}\) has the \emph{uniform strong Feller property} if, for every \(t > 0\) and every bounded Borel function \(\phi \colon \bR \to \bR\), the map \([0, t - h] \times \bR \ni (s, x) \mapsto E^{P_{(s, x)}}[ \phi (X_t) ]\) is uniformly continuous for every \(h \in (0, t)\). 
	\end{enumerate}
\end{definition}
Clearly, the uniform strong Feller property entails the strong Feller property, which itself implies the \(C_b\)--Feller property.

\begin{theorem}[Feller Selection Principle]  \label{theo: Feller selection}
Suppose that the Conditions \ref{cond: convexity}, \ref{cond: LG}, \ref{cond: continuity} and \ref{cond: ellipticity} hold.
    For every \(\psi \in \usc_b(\bR; \mathbb{R})\) and every \(t > 0\), there exists a \(C_0\) and strong Feller family \(\{P_{(s, x)} \colon (s, x) \in \bR_+ \times \mathbb{R}\}\) such that, for all \((s, x)\in \bR_+ \times \mathbb{R}\), \(P_{(s, x)} \in \cK (s, x)\) and 
    \[
    E^{P_{ (s, x) }} \big[ \psi (X_t) \big] = \sup_{P \in \cK (s, x)} E^P \big[ \psi (X_t) \big].
    \]
    In particular, for all \(x \in \bR\), 
    \[
    T_t (\psi) (x) = E^{P_{(0, x)}} \big[ \psi (X_t) \big]. 
    \]
    Moreover, if the Conditions~\ref{cond: bdd} and \ref{cond: uniform ellipticity} hold in addition, then \(\{P_{(s, x)} \colon (s, x) \in \bR_+ \times \mathbb{R}\}\) is also a uniform strong Feller family. 
    
\end{theorem}

The Feller selection principle from Theorem \ref{theo: Feller selection} immediately implies the \(C_0\) and (uniform) strong \(\usc_b\)--Feller properties of the sublinear Markovian semigroup \((T_t)_{t \in \mathbb{R}_+}\), i.e., it proves Theorem~\ref{thm: new very main} (iii) and (iv).

\begin{remark} \label{rem: feller selection not possible}
We highlight the following interesting observation: continuity and linear growth conditions on \(b\) and \(a\) suffice to get the \(\usc_b\)--Feller property in the nonlinear setting but these assumptions do not suffice to select a \(
\usc_b\)--Feller family (equivalently, a \(C_b\)--Feller family by Remark~\ref{rem: schilling}), see \cite[Exercise~12.4.2]{SV} for a counterexample. In particular, the counterexample shows that Theorem~\ref{theo: Feller selection} \emph{fails} without the ellipticity assumption on~\(a\). 
\end{remark}

It turns out that for some classes of input functions \(\psi\) it is possible to select (via an explicit construction) a (time homogeneous) strong Feller family which is \emph{uniform} in \(\psi\) and \(t\). Such a uniform selection principle can be viewed as a \emph{linearization} of the sublinear expectation~\(\mathcal{E}\). 
We discuss this topic in Section~\ref{subsec: linearization} below.

Next, we comment on the proof of the \(\usc_b\)--Feller property.
Recall that for this part of Theorem~\ref{thm: new very main} we do not impose local Lipschitz or ellipticity assumptions.
In our proof we use general theory of correspondences which, from our point of view, provides a rather simple presentation of the argument. 
More precisely, since
\[
T_t (\psi) (x) = \sup_{P \in \cR(x)} E^P \big[ \psi (X_t) \big],
\]
the \(\usc_b\)--Feller property follows from Berge's maximum theorem once the correspondence \(x \mapsto \cR (x)\) is upper hemicontinuous with compact values. We have the following general result:

\begin{theorem} \label{thm: main r}
Suppose that the Conditions \ref{cond: convexity}, \ref{cond: LG} and \ref{cond: continuity} hold. Then, the correspondence
\( x \mapsto \cR (x) \) is upper hemicontinuous with compact values.
\end{theorem}

\begin{remark} \label{rem: lower hemicontinuity}
It is natural to ask whether it is possible to use Theorem \ref{theo: Feller selection} to prove the continuity of the correspondence \(x \mapsto \cR(x)\) and conversely, whether one can establish continuity of \(x \mapsto \cR(x)\) to deduce the \(C_b\)--Feller property from Theorem \ref{theo: Feller selection}. 
By virtue of the generalized version of Berge's maximum theorem \cite[Theorem~1]{berge}, Theorem \ref{theo: Feller selection} implies that the correspondence
\[ 
 x \mapsto \big \{ y \in \bR \colon \exists P \in \cR(x) \text{ such that } y \leq E^P[\psi (X_t)] \big \}, \quad \psi \in C_b(\bR; \bR),\ t > 0,
\]
is lower hemicontinuous. This, however, seems not to give access to the lower hemicontinuity of \( \cR \).

One particular example where continuity of \(\cR\) is rather straightforward to verify 
is the framework of nonlinear L{\'e}vy processes from \cite{neufeld2017nonlinear}, reduced to our path-continuous setting.
That is, the case where the correspondence \( \Theta \) is independent of time and path, i.e.,
\[
\cR(x) = \big\{ P \in \fPas \colon  P \circ X_0^{-1} = \delta_{x}, \ (\llambda \otimes P)\text{-a.e. } (dB^{P} /d\llambda, dC^{P}/d\llambda) \in \Theta \big\},
\]
and the convex and compact set \( \Theta \subset \mathbb{R} \times \mathbb{R_+} \) represents the set of possible means and variances. We refer to \cite[Appendix A]{CN22} for details. It is, however, worth mentioning that compared to Theorem~\ref{theo: Feller selection}, continuity of \( \cR \) is not sufficient to deduce the \(C_0\) or the (uniform) strong \( \usc_b\)--Feller property.

In general, lower hemicontinuity appears to be difficult to verify due to its relation to martingale problems with possibly non-regular coefficients, see \cite[Remark~4.43]{hol16}, \cite[Remark~3.4]{K19} and \cite[Remark~5.4]{K21} for comments in this direction.
In an update of our paper \cite{CN22}, that appeared after the present paper was submitted, we proved lower hemicontinuity of a time- and path-dependent correspondence related to nonlinear continuous semimartingales. The proof from \cite{CN22} relies in a crucial manner on strong existence and pathwise uniqueness properties of SDEs with random coefficients that are implied by the local Lipschitz and linear growth conditions. As already mentioned in the introduction, we are not aware of such existence and uniqueness results under ellipticity assumptions. 
The following is a restatement of \cite[Theorem~4.7]{CN22} tailored to our Markovian situation. 
\begin{theorem} \label{theo: lower hemi}
    Suppose that the Conditions~\ref{cond: convexity}, \ref{cond: LG}, \ref{cond: continuity} and \ref{cond: loc Lip} hold.
    Then, the correspondence
\( x \mapsto \cR (x) \) is lower hemicontinuous.
\end{theorem}
Notice that part (ii) from Theorem~\ref{thm: new very main} follows from part (i) of the same theorem, Theorem~\ref{theo: lower hemi} and Berge's maximum theorem (\cite[Lemma~17.29]{hitchi}).
\end{remark}

\subsection{Linearization} \label{subsec: linearization}

We now present a uniform strong Feller selection principle for two types of nonlinear diffusions and certain classes of input functions.
\begin{condition} [Continuity in Control] \label{cond: continuity in control}
For every \(x \in \mathbb{R}\), \(f \mapsto a(f, x)\) is continuous.
\end{condition}
\begin{condition} [Local H\"older Continuity in Space] \label{cond: local holder}
For every \(M > 0\) there exists a constant \(\C = \C(M) > 0\) such that 
\[
| \sqrt{a} (f, x) - \sqrt{a} (f, y) | \leq \C\, | x - y |^{1/2}
\]
for all \(f \in F\) and \(x,y \in [-M, M]\).
\end{condition}
Let \(\mathbb{G}_{cx}\) be the set of all convex functions \(\psi \colon \mathbb{R} \to \mathbb{R}\) such that 
\begin{equation} \label{eq: df G_cx}
    \exists \C = \C (\psi) > 0, \ m = m(\psi) \in \mathbb{N}\colon \qquad \forall x \in \mathbb{R} \quad |\psi (x)| \leq \C (1 + |x|^m).
\end{equation}

\begin{remark}
Thanks to Lemma \ref{lem: maximal inequality} below, under Condition \ref{cond: LG}, for every \(P \in \cR(x), T \in \bR_+\) and \(\psi \in \mathbb{G}_{cx}\), it holds that \(\psi (X_T) \in L^1(P)\).
\end{remark}

Recall that a (time homogeneous) strong Markov family \(\{P_x \colon x \in \mathbb{R}\}\) is said to be \emph{strongly Feller} if \(x \mapsto E^{P_x} [ \psi (X_t) ]\) is continuous for every \(t > 0\) and every bounded Borel function \(\psi \colon \mathbb{R} \to \mathbb{R}\). 
\begin{theorem}[Uniform Strong Feller Selection Principle] \label{theo: UFSP}
    Suppose that the Conditions \ref{cond: LG}, \ref{cond: ellipticity}, \ref{cond: continuity in control} and \ref{cond: local holder} hold. 
    Furthermore, suppose that \(b \equiv 0\).
    Then, there exists a strong Feller family \(\{P^*_x \colon x \in \mathbb{R}\}\) such that, for all \(x \in \mathbb{R}\), \(P^*_x \in \cR (x)\) and 
        \begin{align}\label{eq: USFSP}
    \mathcal{E}^x ( \psi(X_T) ) = E^{P^*_x} \big[ \psi (X_T) \big] 
    \end{align}
    for all times \(T \in \bR_+\) and \(\psi \in \mathbb{G}_{cx}\) as defined in \eqref{eq: df G_cx}. Moreover, for each \(x \in \mathbb{R}\), \(P^*_x\) is the unique law of a solution process to the SDE
    \[
    d Y_t = \sqrt{a^*} (Y_t) d W_t, \quad Y_0 = x, 
    \]
    where \(W\) is a one-dimensional Brownian motion and \(a^* (y) := \sup \{ a (f, y) \colon f \in F\}\) for \(y \in \mathbb{R}\).
\end{theorem}

The key idea behind Theorem \ref{theo: UFSP} is a stochastic order property. Namely, as is intuitively clear, ordered diffusion coefficients imply a convex stochastic order for one-dimensional distributions. Such a result traces back to the paper \cite{hajek85} whose ideas we also adapt in the proof of Theorem~\ref{theo: UFSP}. Having said all this, it is not hard to believe that a similar result can be proved for the class of continuous increasing functions and nonlinear diffusions with volatility certainty. We think that this application is also of independent interest and therefore we give a precise statement.
\begin{condition} [Continuity of Drift] \label{cond: continuity drift}
\(b\) is continuous.
\end{condition}
\begin{condition} [Certainty, Ellipticity and H\"older Continuity of Volatility] \label{cond: holder}
There exists a \(1/2\)-H\"older continuous function \(a^* \colon \mathbb{R} \to (0, \infty)\) such that \(a (f, x) = a^* (x)\) for all \(f \in F\) and \(x \in \mathbb{R}\).
\end{condition}

\begin{theorem}[Uniform Strong Feller Selection Principle] \label{theo: UFSP 2}
    Suppose that the Conditions \ref{cond: LG}, \ref{cond: continuity drift} and \ref{cond: holder} hold. 
    Then, there exists a strong Feller family \(\{P^*_x \colon x \in \mathbb{R}\}\) such that, for all \(x \in \mathbb{R}\), \(P^*_x \in \cR (x)\) and 
      \begin{align}\label{eq: USFSP 2}
    \mathcal{E}^x ( \psi ) = E^{P^*_x} \big[ \psi \big] 
    \end{align}
    for all bounded increasing (for the pointwise order) Borel functions \(\psi \colon \Omega \to \mathbb{R}\). Moreover, for each \(x \in \mathbb{R}\), \(P^*_x\) is the unique law of a solution process to the SDE
    \[
    d Y_t = b^* (Y_t) dt + \sqrt{a^*} (Y_t) d W_t, \quad Y_0 = x, 
    \]
    where \(W\) is a one-dimensional Brownian motion and \(b^* (y) := \sup \{ b (f, y) \colon f \in F\}\) for \(y \in \mathbb{R}\).
\end{theorem}

Similar to the linear case, nonlinear Markovian semigroups have a close relation to solutions of certain PDEs comparable to Kolmogorov's equation. In the following section we make this relation more precise. In particular, we discuss the connection of the sublinear Markovian semigroup to its \emph{pointwise generator}.

\subsection{A nonlinear Kolmogorov Equation} \label{subsec: viscosity}
Let us start with a formal introduction to the class of nonlinear PDEs under consideration. We fix a finite time horizon \(T > 0\).
For \((t, x,\phi) \in \mathbb{R}_+ \times \mathbb{R} \times C^{1, 2}(\mathbb{R}_+ \times \mathbb{R}; \bR)\), we define
\begin{align*}
    G(t, x, \phi) := \sup \Big\{b (f, x) \partial_x \phi (t,x)
    + \tfrac{1}{2} a (f, x) \partial^2_x\phi (t,x) \colon f \in F \Big\}.
\end{align*}
In our paper \cite{CN22} we proved, under suitable assumptions on \( b\) and \(a\), that the value function
\[
v (t, x) := \sup_{P \in \cR (x)} E^P \big[ \psi (X_{T - t})\big], \quad (t, x) \in [0, T] \times \mathbb{R},
\]
is a \emph{weak-sense viscosity solution} to the nonlinear Kolmogorov type partial differential equation
\begin{equation} \label{eq: PDE}
\begin{cases}   
\partial_t v (t, x) + G (t, x, v) = 0, & \text{for } (t, x) \in [0, T) \times \mathbb{R}, \\
v (T, x) = \psi (x), & \text{for } x \in \bR,
\end{cases}
\end{equation}
where \(\psi  \in C_b(\mathbb{R}; \mathbb{R})\).
Recall that a function \(u \colon [0, T] \times \bR \to \mathbb{R}\) is said to be a \emph{weak sense viscosity subsolution} to \eqref{eq: PDE} if the following two properties hold:
\begin{enumerate}
    \item[\textup{(a)}] \(u(T, \cdot) \leq \psi\);
\item[\textup{(b)}]
\(
\partial_t \phi (t, x) + G (t, x, \phi) \geq 0
\)
for all \(\phi \in  C^{1, 2}([0, T] \times \bR; \bR)\) such that \(\phi \geq u\) and \(\phi (t, x) = u(t, x)\) for some \((t, x) \in [0, T) \times \bR \). 
\end{enumerate}
A \emph{weak sense viscosity supersolution} is obtained by reversing the inequalities. Further, \(u\) is called \emph{weak sense viscosity solution} if it is a weak sense viscosity sub- and supersolution. 
Furthermore,  \( u \) is called \emph{viscosity subsolution} if it is both, a weak sense viscosity subsolution, and upper semicontinuous. The notions of viscosity supersolution and viscosity solution are defined accordingly.

Using Theorem~\ref{thm: new very main}, we are able to present conditions for \(v\) to be a viscosity solution (with regularity). We emphasise that we do not require Lipschitz regularity of the coefficients in all cases.
\begin{theorem} \label{thm: new viscosity no unique}
	Suppose that the Conditions \ref{cond: convexity}, \ref{cond: LG} and \ref{cond: continuity} hold. Further, assume either Condition~\ref{cond: loc Lip} or \ref{cond: ellipticity}. Then, the value function \(v\) is a viscosity solution to the nonlinear PDE \eqref{eq: PDE}.
\end{theorem}

When the local Lipschitz condition is strengthened to a global one, classical comparison results (see, e.g., \cite{hol16,nisio,Pham}) imply that the value function is the unique bounded viscosity solution to the PDE~\eqref{eq: PDE}.

As a referee has pointed out, such a uniqueness result can also be proved by viscosity methods, i.e., without probabilistic arguments for the continuity of the value function. In particular, the continuity of the value function can be established en passant. To detail the strategy, denote the upper and lower envelops of \(v\) by \(v^*\) and \(v_*\), i.e., we set 
\[
v^* (t, x) := \limsup_{ (s, y) \to (t, x) } v (s, y), \qquad v_* (t, x) := \liminf_{ (s, y) \to (t, x) } v (s, y).
\] 
Notice that \(v^*\) is upper semicontinuous while \(v_*\) is lower semicontinuous. 
The key observation is provided by the following lemma. 

\begin{condition}[Lipschitz Continuity in Space] \label{cond: Lipschitz continuity}
    There exists a constant \(\C > 0\) such that 
    \[
    |b(f, x) - b(f, y)| + |\sqrt{a} (f, x) - \sqrt{a} (f, y)| \leq \C\, |x - y|, 
    \]
    for all \(f \in F\) and \(x, y \in \mathbb{R}\).
\end{condition}

\begin{lemma} \label{lem: en upper lower}
Suppose that the Conditions~\ref{cond: LG}, \ref{cond: continuity} and \ref{cond: Lipschitz continuity} hold. Then, \(v^*\) is a viscosity subsolution and \(v_*\) is a viscosity supersolution to the nonlinear PDE \eqref{eq: PDE}.
\end{lemma} 
Under the hypothesis of Lemma~\ref{lem: en upper lower}, a classical (strong) comparison result as given by \cite[Theorem~4.4.5]{Pham} yields that \(v^* \leq v_*\), which entails that
\(v = v^* = v_*.
\)
It follows that \(v\) is a viscosity solution to~\eqref{eq: PDE}, in fact the unique bounded viscosity solution (again by the comparison result). In summary, we proved the following uniqueness result.
\begin{theorem} \label{theo: viscosity with ell}
Suppose that the Conditions \ref{cond: LG}, \ref{cond: continuity} and \ref{cond: Lipschitz continuity} hold.
Then, the value function \(v\) is the unique bounded viscosity solution to the nonlinear PDE \eqref{eq: PDE}.
\end{theorem}
Notice that Theorem~\ref{theo: viscosity with ell} does not require the convexity Condition~\ref{cond: convexity}.

\begin{remark} 
For a sublinear Markovian semigroup \( (T_t)_{t \in \bR_+} \) on the convex cone \( \cH \), its \emph{pointwise infinitesimal generator} \( A \colon  \cD(A) \to \cH \)  is defined by 
\begin{align*}
A(\phi)(x) & := \lim_{t \to 0} \frac{T_t(\phi)(x) - \phi(x)}{t}, \quad x \in \bR, \ \phi \in \cD(A), \\
\cD(A) &:= \Big\{ \phi \in \cH \colon \exists g \in \cH \text{ such that }  \lim_{t \to 0} \frac{T_t(\phi)(x) - \phi(x)}{t} = g(x) \ \ \forall x \in \bR \Big\}.
\end{align*}
For the sublinear Markovian semigroup \( (T_t)_{t \in \bR_+} \) associated to a nonlinear diffusion, following the proof of \cite[Lemmata~7.9, 7.12]{CN22} shows that, under Conditions \ref{cond: convexity}, \ref{cond: LG} and \ref{cond: continuity}, the inclusion   \( C^2_b(\bR; \bR) \subset \cD(A) \) holds with
\[ A (\phi) (x) = \sup \Big\{b (f, x) \phi' (x)
    + \tfrac{1}{2} a (f, x) \phi'' (x) \colon f \in F \Big\}, \quad x \in \bR,
\]
for \( \phi \in  C^2_b(\bR; \bR) \). Hence, the Theorems \ref{thm: new viscosity no unique} and \ref{theo: viscosity with ell} link the sublinear semigroup \((T_t)_{t \in \bR_+}\) to its (pointwise) generator.

\end{remark}

In the following, we show that for convex input functions of polynomial growth the value function solves a \emph{linear} PDE in a unique manner.
More precisely, we prove that for \(\psi \in \mathbb{G}_{cx}\) the value function \(v\) is the unique viscosity and classical solution to the linear PDE
\begin{equation} \label{eq: linear PDE}
\begin{cases}   
\partial_t u (t, x) + \frac{1}{2}a^*(x) \partial_x^2 u(t,x) = 0, & \text{for } (t, x) \in [0, T) \times \mathbb{R}, \\
u (T, x) = \psi (x), & \text{for } x \in \bR,
\end{cases}
\end{equation}
where \(a^*(x) := \sup\{ a(f,x) \colon f \in F \} \).

\begin{condition}[Local Lipschitz Continuity in Space] \label{cond: local Lipschitz continuity}
    For every \(M > 0\) there exists a constant \(\C = \C(M) > 0\) such that 
    \[
    |\sqrt{a} (f, x) - \sqrt{a} (f, y)| \leq \C |x - y|, 
    \]
    for all \(f \in F\) and \(x, y \in [- M, M]\).
\end{condition}

We say that a function \(g \colon [0, T] \times \mathbb{R} \to \mathbb{R}\) is of {\em \(m\)-polynomial growth} if there exists a constant \(\C > 0\) such that \(|g (t, x)| \leq \C (1 + |x|^m)\) for all \((t, x) \in [0, T] \times \mathbb{R}\). 

\begin{corollary} \label{coro: uni con cx}
    Suppose that the Conditions \ref{cond: LG}, \ref{cond: ellipticity}, \ref{cond: continuity in control} and \ref{cond: local Lipschitz continuity} hold. Furthermore, suppose that \(b \equiv 0\).
  If \( \psi \in \mathbb{G}_{cx} \) is of \(m\)-polynomial growth, then the value function \(v\) is the unique viscosity and the unique classical solution 
 of \(m\)-polynomial growth to the linear PDE \eqref{eq: linear PDE}.
  
\end{corollary}

\begin{proof}
Let \( \psi \in \mathbb{G}_{cx}\) be of \(m\)-polynomial growth. Notice that \(\sqrt{a^*}\) is locally Lipschitz continuous and of linear growth thanks to the Conditions \ref{cond: LG} and \ref{cond: local Lipschitz continuity}.
For \(x \in \mathbb{R}\), denote by \(P^*_x\) the unique law of a solution process to the SDE
\[
d Y_t = \sqrt{a^*} (Y_t) d W_t, \quad Y_0 = x, 
\]
where \(W\) is a one-dimensional Brownian motion. Of course, the existence and uniqueness of \(P^*_x\) is classical (see, e.g., Chapter 5 in \cite{KaraShre}).
As the Conditions \ref{cond: LG} and \ref{cond: local Lipschitz continuity} imply Condition \ref{cond: local holder},
Theorem \ref{theo: UFSP} implies 
\[ v(t,x) = E^{P^*_x} \big[ \psi(X_{T-t}) \big], \]
for every \( (t,x) \in [0,T] \times \bR \).
Thanks to this observation, the viscosity part of the corollary follows from \cite[Theorem 1]{FK}. 
As \(v\) is continuous by the previous considerations, it further follows from \cite[Theorem~2.7]{JaTy06} that \(v\) is also a classical solution to \eqref{eq: linear PDE} and finally, the uniqueness among classical solutions of \(m\)-polynomical growth follows from \cite[Corollary 6.4.4]{friedman75}.
The proof is complete. 
\end{proof}

\subsection{Relation to Nisio semigroups and the \(\uc_b\)--Feller property} \label{sec: nisio}
Another approach to sublinear semigroups is discussed in the recent paper \cite{NR}.
There, the authors start with a family \((S^\lambda)_{\lambda \in \Lambda}\) of linear semigroups on the space \(\uc_\kappa(\bR;\bR)\) of weighted uniformly continuous functions, where the weight function \(\kappa \colon \bR \to (0, \infty)\) is assumed to be bounded and continuous.\footnote{The paper \cite{NR} allows for more general state spaces, but for the sake of comparison we focus on the one-dimensional case.}
They present general conditions for the {\em upper semigroup envelope} \((\mathscr{S}_t)_{t \in \bR_+}\) (sometimes also called {\em Nisio semigroup}) of the family \((S^\lambda)_{\lambda \in \Lambda}\) to be a sublinear semigroup on \(\uc_\kappa(\bR;\bR)\). Furthermore, they establish a viscosity theory for their framework. Using this theory, certain Nisio semigroups can be related to the sublinear semigroup \((T_t)_{t \in \bR_+}\) that is studied in this paper. 
In this section, choosing \(\kappa =1\), we relate the Nisio semigroup from \cite{NR} to nonlinear diffusions with bounded Lipschitz continuous coefficients.
The connection to the Nisio semigroup also provides conditions for the \(\uc_b\)--Feller property of the semigroup \((T_t)_{t \in \bR_+}\) that are different from those in Theorem~\ref{thm: new very main}. The \(\uc_b\)--Feller property was already observed in \cite{denk2020semigroup,K19} for nonlinear L\'evy processes. 

The space of bounded Lipschitz continuous functions from \(\bR\) into \(\bR\) is denoted by \(\lip_b (\bR; \bR)\) and the corresponding Lipschitz norm is denoted by~\(\|\cdot\|_{\lip}\).

\begin{remark}
    In case the weight function \(\kappa\) vanishes at infinity, the space \(\uc_\kappa(\bR;\bR)\) has the explicit description
    \[
    \uc_\kappa(\bR;\bR) = \big\{ u \in C(\bR; \bR) \colon u \kappa  \in C_0(\bR;\bR) \big\},
    \]
    cf. \cite[Remark~5.3 (b)]{NR}. Notice that in this case \(C_b(\bR; \bR) \subset \uc_\kappa(\bR;\bR)\) and consequently, the Nisio semigroup \((\mathscr{S}_t)_{t \in \bR_+}\) has the \(C_b\)--Feller property. 
    In particular, \cite[Section 6.3]{NR} provides examples for nonlinear Markov processes whose associated Nisio semigroups entail the \(C_b\)--Feller property that go beyond the L\'evy case that was discussed in \cite{denk2020semigroup,K19}.

    If the weight function is bounded from below, then \(\uc_\kappa(\bR;\bR)\) coincides with the space \(\uc_b(\bR;\bR)\) of bounded uniformly continuous functions.
    Extension of the semigroup to \(C_b(\bR;\bR)\) is then ensured by continuity from above on \(\lip_b(\bR;\bR)\), cf. \cite[Remark~5.3 (c)]{NR}. Under boundedness assumptions, this is verified for the class of nonlinear L\'evy processes in \cite[Example~7.2]{NR}, see also \cite[Proposition~2.8]{denk2020semigroup} and \cite[Proposition~4.10]{K21}.

    Throughout the paper \cite{NR}, the following assumption is imposed on the family of semigroups \((S^\lambda)_{\lambda \in \Lambda}\):
    \begin{align} \label{eq: A2 from NR}
    \exists\, \alpha, \beta \in \bR \colon \qquad \|S^\lambda_t (u)\|_\kappa \leq e^{\alpha t} \|u\|_\kappa, \qquad \|S^\lambda_t (u)\|_{\lip} \leq e^{\beta t} \|u\|_{\lip}
    \end{align}
    for all \(u \in \lip_b (\bR; \bR)\), \(\lambda \in \Lambda\) and \(t \in \bR_+\).
    This allows to propagate the \(\uc_\kappa\)--Feller Property of \(S^\lambda\), \(\lambda \in \Lambda\), to the Nisio semigroup \((\mathscr{S}_t)_{t \in \bR_+}\).
    Example \ref{ex: A2 fails} below shows that the second part of \eqref{eq: A2 from NR} fails for semigroups related to SDEs under mere continuity and linear growth assumptions on the drift and volatility coefficients. 
    Below, we verify \eqref{eq: A2 from NR} for a family of semigroups related to SDEs under Lipschitz and H\"older conditions. 
    
    It is worth mentioning that the paper \cite{NR} establishes a stochastic representation for Nisio semigroups that are continuous from above on \(\lip_b(\bR;\bR)\).
    To the best of our knowledge, it is not known in general that this representation coincides with \((T_t)_{t \in \bR_+}\). For the L\'evy case, i.e., sublinear Markovian convolution semigroups, this was shown in \cite[Theorem~6.4]{K21} by identifying the associated generators, and uniqueness of the corresponding evolution equation in the viscosity sense.
    Using this approach, we verify below that \((T_t)_{t \in \bR_+}\) agrees with the Nisio semigroup on \(\uc_b(\bR;\bR)\) under Lipschitz and boundedness conditions.
    In particular, this gives access to the \(\uc_b\)--Feller property of \((T_t)_{t \in \bR_+}\).
\end{remark}
 
\smallskip
In order to define the Nisio semigroup, we will impose the following condition.
\begin{condition} \label{cond: fix f cond}
For every \(f \in F\), the map \(x \mapsto b(f, x)\) is of linear growth, and the map \(x \mapsto \sqrt{a} (f, x)\) is of linear growth and locally H\"older continuous with exponent \(1 / 2\). Furthermore, there exists a constant \(\beta > 0\) such that 
\[
| b(f, x) - b (f, y) | \leq \beta\, | x- y|
\]
for all \(f \in F\) and \(x, y \in \bR\).
\end{condition}

Let \(\mathbb{B} = (\Sigma, \mathcal{A}, (\mathcal{A}_t)_{t \in \bR_+}, P)\) be a filtered probability space that supports a one-dimensional standard Brownian motion \(W\). 
In case Condition~\ref{cond: fix f cond} holds, for every \((f, x) \in F \times \bR\), there exists a continuous adapted process \(Y^{f, x}\) on the stochastic basis \(\mathbb{B}\) with dynamics
\[
d Y_t^{f, x} = b (f, Y^{f, x}_t) dt + \sqrt{a} (f, Y^{f, x}_t) d W_t, \quad Y^{f, x}_0 = x, 
\]
cf. \cite[Chapter~5]{KaraShre} or \cite[Chapter~IX]{RY}. 
In particular, martingale problem arguments (see, e.g., \cite[Theorem~5.4.20, Remark~5.4.21]{KaraShre}) show that each \(\{P \circ (Y^{f, x})^{-1} \colon x \in \bR\}\) is a strong Markov family. 
Hence, the operators
\[
S^f_t (u) (x) := E^P \big[ u (Y^{f, x}_t) \big], \quad u \in \uc_b (\bR; \bR), 
\]
satisfy the semigroup property \(S^f_{t + s} = S^f_t S^f_s\) for \(s, t \in \bR_+\). In fact, as the following lemma shows, each \((S^f_t)_{t \in \mathbb{R}_+}\) is a linear semigroup on \(\uc_b (\bR; \bR)\), whose (pointwise) generator \(A^f\) satisfies
\[
A^f (u)(x) = b(f,x) u'(x) + \tfrac{1}{2} a(f,x) u''(x), \quad u \in C^2_b(\bR; \bR).
\]
\begin{lemma} \label{lem: Sf semigroup}
Suppose that Condition~\ref{cond: fix f cond} holds. Then, for every \(f \in F\), the family \((S^f_t)_{t \in \mathbb{R}_+}\) is a linear semigroup on \(\uc_b (\bR; \bR)\) (that is also monotone and continuous from below in the sense of \cite[Definition~1.1]{NR}).
\end{lemma}

Following \cite{NR}, we define the nonlinear operator
\[
\mathcal{J}_t := \sup_{f \in F} S^f_t, \quad t \in \bR_+.
\]
Let \(\Pi_t\) be the set of finite partitions \(0 = t_0 < \dots < t_m = t\) of the interval \([0, t]\). For a partition \(\pi = \{t_0, \dots, t_m\}\), we set 
\[
\mathcal{J}_\pi := \mathcal{J}_{t_1 - t_0} \cdots \mathcal{J}_{t_m - t_{m - 1}}.
\]
Finally, we define 
\[
\mathscr{S}_t := \sup_{\pi \in \Pi_t} \mathcal{J}_{\pi}, \quad t \in \bR_+.
\]
The family \((\mathscr{S}_t)_{t \in \bR_+}\) is called the {\em Nisio semigroup} associated to \(\{ (S^f_t)_{t \in \bR_+} \colon f \in F\}\). 
The next result shows that \((\mathscr{S}_t)_{t \in \bR_+}\) deserves to be called ``semigroup''.  

\begin{proposition} \label{prop: first NR}
    Suppose that Condition~\ref{cond: fix f cond} holds.
    \begin{enumerate}
\item[\textup{(i)}]
\((\mathscr{S}_t)_{t \in \bR_+}\) is a sublinear Markovian semigroup on \(\uc_b (\bR; \bR)\) that is continuous from below, i.e., for every \(t \in \bR_+\) and any sequence \((u^n)_{n = 0}^\infty \subset \uc_b (\bR; \bR)\) with \(u^n \nearrow u^0\) pointwise it holds that \(\mathscr{S}_t (u^n) \nearrow \mathscr{S}_t (u^0)\) pointwise as \(n \to \infty\).
\end{enumerate}
Assume in addition that Condition~\ref{cond: bdd} holds.
\begin{enumerate}
\item[\textup{(ii)}]
\((\mathscr{S}_t)_{t \in \bR_+}\) is strongly continuous, i.e., \(t \mapsto \mathscr{S}_t (u)\) is continuous from \(\bR_+\) into \(\uc_b (\bR; \bR)\) for every \(u \in \uc_b (\bR; \bR)\).
\item[\textup{(iii)}] \((\mathscr{S}_t)_{t \in \bR_+}\) is continuous from above on \(\lip_b (\bR; \bR)\), i.e., for every \(t \in \bR_+\) and any sequence \((u^n)_{n = 1}^\infty \subset \lip_b (\bR; \bR)\) with \(u^n \searrow 0\) pointwise it holds that \(\mathscr{S}_t (u^n) \searrow 0\) pointwise as \(n \to \infty\).
    \end{enumerate}
\end{proposition}

Continuity from below and above of the semigroup \((\mathscr{S}_t)_{t \in \bR_+}\) allows us to extend the operators \(\mathscr{S}_t\) to \(C_b(\bR;\bR)\). We record this in the following proposition.

\begin{proposition}\label{prop: extension}
Suppose that the Conditions~\ref{cond: bdd} and \ref{cond: fix f cond} hold.
There exists a unique sublinear Markovian semigroup 
\((\widehat{\mathscr{S}}_t)_{t \in \bR_+}\) on \(C_b(\bR;\bR)\) that is continuous from above and
such that \(\widehat{\mathscr{S}}_t =\mathscr{S}_t\) on \(\uc_b(\bR; \bR)\).
\end{proposition}
    Recall that a standing assumption (when the growth function \(\kappa\) is taken to be constantly one, as we do here) in the paper \cite{NR} is the following: there are constants \(\alpha, \beta \in \bR\) such that 
    \begin{align*} 
    \|S^f_t (u)\|_\infty \leq e^{\alpha t} \|u\|_\infty, \qquad \|S^f_t (u)\|_{\lip} \leq e^{\beta t} \|u\|_{\lip}
    \end{align*}
    for all \(u \in \lip_b (\bR; \bR)\), \(f \in F\) and \(t \in \bR_+\). In the proof of Proposition~\ref{prop: first NR}, we show that this assumption holds under Condition~\ref{cond: fix f cond}. The following example shows that the second estimate fails under mere continuity and linear growth conditions on the coefficients.

\begin{example} \label{ex: A2 fails}
Let \(B\) be a one-dimensional standard Brownian motion starting in zero. 
The semigroup \[
S_t (u) (x) := E \big[ u (Y^x_t) \big], \quad Y^x := (B + x^{1/3})^3, \] 
does not satisfy \(\|S_t (u)\|_{\lip} \leq e^{\beta t}\|u\|_{\lip}\) for all \(u \in \lip_b(\bR; \bR)\) and \(t \in \bR_+\). Indeed, this follows\footnote{The argument is indirect: If \(\|S_1 (u)\|_{\lip} \leq e^{\beta} \|u\|_{\lip}\) holds for all \(u \in \lip_b(\bR; \bR)\), then the same inequality must also hold for \(u = \on{id}\).} from the observation that
\[
S_1 ( \on{id})(x) = E \big[ (B_1 + x^{1/3})^3 \big] = x + 3 x^{1/3}.
\]
Furthermore, It\^o's formula yields that 
\[
d Y^x_t = 3 (Y^x_t)^{1/3} dt + 3 (Y^x_t)^{2/3} d B_t, \quad Y^x_0 = x, 
\]
which shows that the generator of \((S_t)_{t \in \bR_+}\) satisfies 
\[
A (u) (x) = 3 x^{1/3} u' (x) + \tfrac{9}{2} x^{4/3} u'' (x), \quad u \in C^2_b (\bR; \bR).
\]
It is interesting to note that the class \(C^2_b (\bR; \bR)\) does not suffice to characterize the generator uniquely (cf. \cite[Exercise~5.2.17]{KaraShre}). 
\end{example}

By Theorem~\ref{thm: new viscosity no unique} and results from \cite[Section~4]{NR}, both the Nisio semigroup \((\mathscr{S}_t)_{t \in \bR_+}\) and our semigroup \((T_t)_{t \in \bR_+}\) can be related to the same \emph{generator equation}
\begin{align} \label{eq: real generator eq}
 \begin{cases} \partial_t u (t, x)  = \sup_{f \in F} A^f(u) (t, x) = G (t, x, u), & \text{for } (t, x) \in [0, T) \times \bR,\\
 u (T, x) = \psi (x), & \text{for } x \in \mathbb{R}.
 \end{cases}
\end{align}
Under an additional global Lipschitz condition that ensures uniqueness for \eqref{eq: real generator eq}, see Theorem~\ref{theo: viscosity with ell} above, we can prove that \((\mathscr{S}_t)_{t \in \bR_+} = (T_t)_{t \in \bR_+}\).

\begin{theorem} \label{thm: nisio}
    Suppose that the Conditions~\ref{cond: bdd}, \ref{cond: continuity} and \ref{cond: Lipschitz continuity} hold. Then, \((\mathscr{S}_t)_{t \in \bR_+} = (T_t)_{t \in \bR_+}\) on \(\uc_b (\bR; \bR)\). In particular, \((T_t)_{t \in \bR_+}\) has the \(\uc_b\)--Feller property.
\end{theorem}

The semigroup \((T_t)_{t \in \bR_+}\) is continuous from above on \(C_b(\bR; \bR)\) by \cite[Proposition~3.5]{denk2018} and Proposition~\ref{prop: compactness} below.
Hence, combining Theorem \ref{thm: nisio} with Theorem~\ref{theo: viscosity with ell} and Proposition~\ref{prop: extension}, we obtain the following corollary.
\begin{corollary}
     Suppose that the Conditions~\ref{cond: bdd}, \ref{cond: continuity} and \ref{cond: Lipschitz continuity} hold.
     Then, \((T_t)_{t \in \bR_+}\) is the unique extension of \((\mathscr{S}_t)_{t \in \bR_+}\) to \(C_b(\bR;\bR)\) that is continuous from above.
\end{corollary}

When arguing based on the generator equation, it might be difficult to relate our framework to the one from \cite{NR} without suitable regularity conditions on the coefficients \(b\) and \(a\).
In the examples from \cite[Section~6.3]{NR}, \(b\) and \(a\) satisfy global Lipschitz conditions (comparable to Condition~\ref{cond: Lipschitz continuity}).


\section{The Regularity of \( \cR \)} \label{sec: regularity r}
In this section we prove that \(\cR\) is upper hemicontinuous with compact values. 
\subsection{Some Preparations}
We start with a few properties of the correspondence \(\Theta\).
\begin{lemma} \label{lem: continuity theta}
Suppose that Condition \ref{cond: continuity} holds.
Then, the correspondence \(x \mapsto \Theta (x)\) is compact-valued and continuous.
\end{lemma}

The previous lemma is a direct consequence of the following general observation.

\begin{lemma} \label{lem: continuity abstract}
Let \(F, E \) and \( D \) be topological spaces. If
\( g \colon F \times E \to D \) is continuous, and \(F \) is compact, then the correspondence \( \varphi \colon E \twoheadrightarrow D \) defined by \( \varphi(x) := g(F,x) \) is compact-valued and continuous.
\end{lemma}
\begin{proof}
By construction, \( \varphi \) has compact values. Regarding the continuity, note that \( \varphi \) is the composition of the correspondence
\( E \ni x \mapsto F \times \{ x \} \)
 and the (single-valued) correspondence 
\( F \times E \ni (f,x) \mapsto \{ g(f,x) \} \). 
While the latter correspondence is continuous due to continuity of \( g \), the former is continuous being the finite product of 
compact-valued continuous correspondences, cf. \cite[Theorem 17.28]{hitchi}.
Thus, continuity of \( \varphi \) follows from \cite[Theorem 17.23]{hitchi}.
\end{proof}

The next lemma is an auxiliary result regarding a large class of continuous correspondences.

\begin{lemma} \label{lem: continuity interval}
Let \( E \) be a topological space, and let \(f, g \colon E \to \mathbb{R} \)
be continuous functions with \( f(x) \leq g(x) \) for every \( x \in E\).
Then, the correspondence \( E \ni x \mapsto [f(x), g(x)] \) is continuous with compact values.
\end{lemma}
\begin{proof}
By \cite[Theorem 17.15]{hitchi}, continuity of the correspondence  \( E \ni x \mapsto [f(x), g(x)] \) is equivalent to continuity of the function
\( E \ni x \mapsto [f(x), g(x)] \in \mathcal{K}(\bR)\),
where \( \mathcal{K}(\bR) \), the collection of nonempty compact subsets of \( \bR\), is equipped with the Hausdorff metric \( d_H\).
As 
\[ d_H([a,b], [c,d]) = \max \big\{ | a - c|, |b-d| \big\},
\]
for every \(a \leq b\), \(c \leq d \),
continuity of \( E \ni x \mapsto [f(x), g(x)] \in \mathcal{K}(\bR)\) follows from continuity of \(f \) and~\(g \).
\end{proof}

For a subset \(G\) of a locally convex space we denote by \(\oconv G\) the closure of the convex hull generated by \(G\).
\begin{lemma} \label{lem: continuity result for theta}
Let \(D\) be a locally convex space and let \(\varphi \colon \mathbb{R}_+ \twoheadrightarrow D\) be an upper hemicontinuous correspondence with convex and compact values such that \(\oconv \varphi([t, t + 1])\) is compact.  Then, for every \(t \in \mathbb{R}_+\), we have 
\[
(a_n)_{n \in \mathbb{N}} \subset (0, 1], \ a_n \to 0\ \Longrightarrow \ \bigcap_{m \in \mathbb{N}} \oconv  \varphi ([t, t + a_m]) \subset \varphi (t).
\]
\end{lemma}
\begin{proof}
Notice that \( [0,1] \ni s \mapsto \psi (s) := \varphi ([t, t + s])\) is upper hemicontinuous as a composition of the continuous (Lemma \ref{lem: continuity interval}) correspondence \(s \mapsto [t, t + s] \) and the upper hemicontinuous correspondence~\(\varphi\), see \cite[Theorem 17.23]{hitchi}. By our assumption, \(\oconv \psi (s)\) is compact, being a closed subset of the compact set \(\oconv \varphi([t, t + 1])\), and we deduce from \cite[Theorem~17.35]{hitchi} that \(s \mapsto \phi (s) := \oconv \psi(s)\) is upper hemicontinuous. Take \(x \in \bigcap_{m \in \mathbb{N}} \phi (a_m)\). Then, for each \(m \in \mathbb{N}\), \((a_m, x) \in \on{gr} \phi\) and hence, by \cite[Theorem~17.16]{hitchi}, as \(\phi\) is compact-valued, the upper hemicontinuity and \(a_m \to 0\) imply that \(x \in \phi(0)\). Finally, observing that \(\phi (0) = \varphi(t)\) completes the proof.
\end{proof}

\begin{lemma} \label{lem: oconv inclusion theta}
Suppose that the Conditions \ref{cond: convexity} and \ref{cond: continuity} hold. Then, 
\[
\bigcap_{m \in \mathbb{N}} \oconv \Theta (\omega([t, t + 1/m])) \subset \Theta (\omega(t))
\]
for all \((t, \omega) \in \of 0, \infty\of\).
\end{lemma}
\begin{proof}
By Lemma \ref{lem: continuity theta} and Condition \ref{cond: convexity}, \(t \mapsto \Theta (\omega (t))\) is continuous with compact and convex values. Furthermore, by the continuity of \(b\) and \(a\), i.e., Condition \ref{cond: continuity}, the set \(\Theta (\omega([t, t + 1]))\) is compact. Consequently, as in completely metrizable locally convex spaces the closed convex hull of a compact set is itself compact (\cite[Theorem 5.35]{hitchi}), we conclude that \(\oconv \Theta ( \omega([t, t + 1]))\) is compact. Finally, the claim follows from Lemma~\ref{lem: continuity result for theta}.
\end{proof}

\subsection{\(\cR\) is compact-valued}
We start with a first auxiliary observation. For \(M > 0\) and \(\omega \in \Omega\), define 
\[
\tau_M (\omega) := \inf \{t \geq 0 \colon |\omega (t)| \geq M\} \wedge M.
\]
Furthermore, for \(\omega = (\omega^{(1)}, \omega^{(2)}) \in \Omega \times \Omega\), we set 
\[
\zeta_M (\omega) := \sup \Big\{ \frac{|\omega^{(2)} (t \wedge \tau_M (\omega^{(1)})) - \omega^{(2)}(s \wedge \tau_M(\omega^{(1)}))|}{t - s} \colon 0 \leq s < t\Big\}.
\]
\begin{lemma} \label{lem: Lipschitz Contant}
Let \(P\) be a Borel probability measure on \(\Omega \times \Omega\).
There exists a set \(D \subset \mathbb{R}_+\) with countable complement such that for every \(M \in D\) 
there exists a \(P\)-null set \(N = N(M)\) such that \(\zeta_M\) is lower semicontinuous at all \(\omega \not \in N\).
\end{lemma}
\begin{proof}
Due to \cite[Lemma 11.1.2]{SV}, for all but countably many \(M \in \mathbb{R}_+\), there exists a \(P\)-null set \(N = N(M)\) such that \(\omega \mapsto \tau_M (\omega^{(1)})\) is continuous at all \(\omega \not \in N\). Take such an \(M \in \mathbb{R}_+\) and \(\omega \not \in N\). Furthermore, let \((\omega_n)_{n \in \mathbb{N}} \subset \Omega \times \Omega\) be such that \(\omega_n \to \omega\). Then, 
\begin{align*}
\zeta_M (\omega) &= \sup \Big\{ \liminf_{n \to \infty} \frac{| \omega_n^{(2)} (t \wedge \tau_M(\omega^{(1)}_n)) - \omega^{(2)}_n (s \wedge \tau^{(1)}_M (\omega_n))|}{t - s} \colon 0 \leq s < t\Big\} 
\leq \liminf_{n \to \infty} \zeta_M (\omega_n).
\end{align*}
The proof is complete.
\end{proof}
The next lemma follows similar to the proof (of the second part) of \cite[Lemma~7.4]{CN22}, see also \cite[Problem~5.3.15]{KaraShre}. We skip the details for brevity.
\begin{lemma} \label{lem: maximal inequality}
Suppose that Condition \ref{cond: LG} holds.
For every bounded set \(K \subset \mathbb{R}\) and \(T, m > 0\), there exists a constant \(\C > 0\) such that, for all \(s, t \in [0, T]\),
\[
\sup_{x \in K} \sup_{P \in \cR(x)} E^P \Big[ \sup_{r \in [0, T]} |X_r|^m \Big] < \infty, \qquad \sup_{x \in K} \sup_{P \in \cR(x)} E^P \big[ |X_{t} - X_{s}|^m \big] \leq \C |t - s|^{m/2}.
\]
\end{lemma}

The next results extend \cite[Theorem 4.41]{hol16} and \cite[Theorem 2.5]{neufeld} beyond the case where \(b\) and \(a\) are uniformly bounded and globally Lipschitz continuous.

\begin{proposition} \label{prop: closedness}
Suppose that the Conditions \ref{cond: convexity}, \ref{cond: LG} and \ref{cond: continuity} hold.
The set 
\begin{align*}
\mathfrak{P}(\Theta) := \big\{ P \in \fPas \colon P &\circ X_0^{-1} \in \{\delta_x \colon x \in \bR\}, \ (\llambda \otimes P)\text{-a.e. } (dB^{P} /d\llambda, dC^{P}/d\llambda) \in \Theta(X)  \big\}
\end{align*}
is closed in \(\mathfrak{P}(\Omega)\). 
\end{proposition}
\begin{proof}
Let \((P^n)_{n \in \mathbb{N}} \subset \mathfrak{P}(\Theta)\) be such that \(P^n \to P\) weakly. 
By definition of \(\mathfrak{P}(\Theta)\), for every \(n \in \mathbb{N}\), there exists a point \(x^n \in \mathbb{R}\) such that \(P^n \circ X_0^{-1} = \delta_{x^n}\) and hence, \(P^n \in \cR(x^n)\). Since \(P^n \to P\) and \(\{\delta_x \colon x \in \mathbb{R}\}\) is closed (\cite[Theorem 15.8]{hitchi}), there exists a \(x^0 \in \mathbb{R}\) such that \(P \circ X^{-1}_0 = \delta_{x^0}\). In particular, \(x^n \to x^0\) and the set \(U := \{x^n \colon n \in \mathbb{N}\}\) is bounded.
It remains to prove that \(P\in \fPas\) with differential characteristics in \(\Theta\).
The proof of this is split into four steps.
In order to execute our program, we need a last bit of auxiliary notation.
For each \(n \in \mathbb{N}\), denote the \(P^n\)-characteristics of \(X\) by \((B^n, C^n)\).
Define \(\Omega^* := \Omega \times \Omega \times \Omega\) and denote the coordinate process on \(\Omega^*\) by \(Y = (Y^{(1)}, Y^{(2)}, Y^{(3)})\). Further, set \(\cF^* := \sigma (Y_s, s \geq 0)\) and let \(\F^* = (\cF^*_s)_{s \geq 0}\) be the right-continuous filtration generated by \(Y\).

\emph{Step 1.} We start by showing that \(\{P^n \circ (X, B^n, C^n)^{-1} \colon n \in \mathbb{N}\}\) is tight on the space \((\Omega^*, \cF^*)\).
Since \(P^n \to P\), it suffices to prove tightness of \(\{P^n \circ (B^n, C^n)^{-1} \colon n \in \mathbb{N}\}\). We use Aldous' tightness criterion (\cite[Theorem~VI.4.5]{JS}), i.e., we show the following two conditions:
\begin{enumerate}
    \item[(a)]
    for every \(N, \varepsilon > 0\) there exists a \(K \in \mathbb{R}_+\) such that
    \[
    \sup_{n \in \mathbb{N}} P^n \Big( \sup_{s \in [0, N]} |B^n_s| + \sup_{s \in [0, N]} |C^n_s| \geq K \Big) \leq \varepsilon;
    \]
    \item[(b)] 
    for every \(N, \varepsilon > 0\),
    \[
    \lim_{\theta \searrow 0} \limsup_{n \to \infty} \sup \big\{P^n (|B^n_T - B^n_S| + |C^n_T - C^n_S| \geq \varepsilon) \big\} = 0, 
    \]
    where the \(\sup\) is taken over all stopping times \(S, T \leq N\) such that \(S \leq T \leq S + \theta\).
\end{enumerate}
For a moment, let us fix \(N > 0\).
Thanks to Lemma \ref{lem: maximal inequality}, recalling that \(P^n \in \cR(x^n)\) and that \(U = \{x^n \colon n \in \mathbb{N}\}\) is bounded, we have
\begin{align} \label{eq: second moment bound}
\sup_{n \in \mathbb{N}} E^{P^n} \Big[ \sup_{s \in [0, N]} |X_s|^2 \Big] 
\leq \sup_{x \in U} \sup_{P \in \cR (x)} E^P \Big[ \sup_{s \in [0, N]} |X_s|^2 \Big]
< \infty.
\end{align}
Now, by the definition of \(\Theta\) and the linear growth assumption (Condition \ref{cond: LG}), we get that \(P^n\)-a.s.
\[
\sup_{s \in [0, N]} |B^n_s| + \sup_{s \in [0, N]} |C^n_s| \leq \C \Big( 1 + \sup_{s \in [0, N]} |X_s|^2 \Big),
\]
where the constant \(\C>0\) might depend on \(N\) but is independent of \(n\).
By virtue of \eqref{eq: second moment bound}, this bound immediately yields (a). For (b), take two stopping times \(S, T \leq N\) such that \(S \leq T \leq S + \theta\) for some \(\theta > 0\). Then, using again the definition of \(\Theta\) and the linear growth assumptions, we get \(P^n\)-a.s.
\[
|B^n_T - B^n_S| + |C^n_T - C^n_S| \leq \C (T - S) \Big( 1 + \sup_{s \in [0, N]} |X_s|^2 \Big) \leq \C \theta \Big( 1 + \sup_{s \in [0, N]} |X_s|^2 \Big),
\]
which yields (b) by virtue of \eqref{eq: second moment bound}. We conclude that \(\{P^n \circ (X, B^n, C^n)^{-1} \colon n \in \mathbb{N}\}\) is tight. 
Up to passing to a subsequence, from now on we assume that \(P^n \circ (X, B^n, C^n)^{-1} \to Q\) weakly, where \(Q\) is a probability measure on \((\Omega^*, \cF^*)\).

\emph{Step 2.}  Next, we show that \(Y^{(2)}\) and \(Y^{(3)}\) are \(Q\)-a.s. locally absolutely continuous.
Thanks to Lemma~\ref{lem: Lipschitz Contant}, there exists a dense set \(D \subset \mathbb{R}_+\) such that, for every \(M \in D\), the map \(\zeta_M\) is \(Q \circ (Y^{(1)}, Y^{(2)})^{-1}\)-a.s. lower semicontinuous. By virtue of Condition \ref{cond: LG} and the definition of \(\tau_M\), for every \(M \in D\) there exists a constant \(\C = \C (M) > 0\) such that \(P^n(\zeta_M (X, B^n) \leq \C) = 1\) for all \(n \in \mathbb{N}\). As \(\zeta_M\) is \(Q \circ (Y^{(1)}, Y^{(2)})^{-1}\)-a.s. lower semicontinuous, \cite[Example 17, p. 73]{pollard} yields that 
\[
0 = \liminf_{n \to \infty} P^n (\zeta_M (X, B^n) > \C) \geq Q (\zeta_M (Y^{(1)}, Y^{(2)}) > \C). 
\]
Further, since \(D\) is dense in \(\mathbb{R}_+\), we conclude that \(Q\)-a.s. \(Y^{(2)}\) is locally Lipschitz continuous, i.e., in particular locally absolutely continuous. Similarly, we get that \(Y^{(3)}\) is \(Q\)-a.s. locally Lipschitz and hence, locally absolutely continuous. 

\emph{Step 3.}
We define a map \(
\Phi \colon \Omega^* \to \Omega
\)
by \(\Phi (\omega^{(1)}, \omega^{(2)}, \omega^{(3)}) := \omega^{(1)}\). Clearly, we have \(Q \circ \Phi^{-1} = P\) and \(Y^{(1)} = X \circ \Phi\). 
In this step, we prove that \((\llambda \otimes Q)\)-a.e. \((dY^{(2)} /d \llambda, dY^{(3)}/ d \llambda) \in \Theta(Y^{(1)})\). For a moment, let us fix \(m \in \mathbb{N}\).
By virtue of \cite[Corollary 8, p. 48]{diestel}, \(P^n\)-a.s. for \(\llambda\)-a.a. \(t \in \mathbb{R}_+\), we have 
\begin{equation}\label{eq: P as inclusion theta}
\begin{split}
m (B^n_{t + 1/m} - B^n_t, C^n_{t + 1/m} - C^n_t) &\in \oconv ( dB^n / d \llambda, d C^n / d \llambda) ([ t, t + 1/m ]) \\&\subset \oconv \Theta (X([t, t + 1/m])).
\end{split}
\end{equation}
By Skorokhod's coupling theorem, with little abuse of notation, there exist random variables \[(X^0, B^0, C^0), (X^1, B^1, C^1), (X^2, B^2, C^2), \dots\] defined on some probability space \((\Sigma, \mathcal{G}, R)\) such that \((X^0, B^0, C^0)\) has distribution \(Q\), \((X^n, B^n, C^n)\) has distribution \(P^n\circ (X, B^n, C^n)^{-1}\) and \(R\)-a.s. \((X^n, B^n, C^n) \to (X^0, B^0, C^0)\) in the local uniform topology. We deduce from Lemma \ref{lem: continuity abstract} that the correspondence \(\omega \mapsto \Theta (\omega([t, t + 1 /m]))\) is continuous  for every \(t \in \bR_+\). Furthermore,  for every \(t \in \bR_+\), as \(\oconv \Theta (\omega([t, t + 1/m]))\) is compact (by \cite[Theorem 5.35]{hitchi}) for every \(\omega \in \Omega\), it follows from \cite[Theorem 17.35]{hitchi} that the correspondence \(\omega \mapsto \oconv \Theta (\omega([t, t + 1/m]))\) is upper hemicontinuous and compact-valued. Thus, by virtue of \eqref{eq: P as inclusion theta} and \cite[Theorem 17.20]{hitchi}, we get, \(R\)-a.s. for \(\llambda\)-a.a. \(t \in \bR_+\), that 
\[
m (B^0_{t + 1/m} - B^0_t, C^0_{t + 1/m} - C^0_t) \in \oconv \Theta ( X^0([t, t + 1/m])).
\]
Notice that \((\llambda \otimes R)\)-a.e.
\[
(d B^0 / d \llambda, d C^0 / d \llambda) = \lim_{m \to \infty} m (B^0_{\cdot + 1/m} - B^0_\cdot, C^0_{\cdot + 1/m} - C^0_\cdot).
\]
Now, using Lemma \ref{lem: oconv inclusion theta}, we conclude that \(R\)-a.s. for \(\llambda\)-a.a. \(t \in \mathbb{R}_+\)
\[
(d B^0 / d \llambda, d C^0 / d \llambda) (t) \in \bigcap_{m \in \mathbb{N}} \oconv \Theta (X^0([t, t + 1/m])) \subset \Theta (X^0_t).
\]
This shows that \( (\llambda \otimes Q)\)-a.e. \((dY^{(2)} /d \llambda, dY^{(3)}/ d \llambda) \in \Theta (Y^{(1)})\).

\emph{Step 4.} In the final step of the proof, we show that \(P \in \fPas\) and we relate \((Y^{(2)}, Y^{(3)})\) to the \(P\)-semimartingale characteristics of the coordinate process.
Thanks to \cite[Lemma 11.1.2]{SV}, there exists a dense set \(D \subset \bR_+\) such that \(\tau_M \circ \Phi\) is \(Q\)-a.s. continuous for all \(M \in D\). Take some \(M \in D\). Since \(P^n \in \fPas\), it follows from the definition of the first characteristic that the process \(X_{\cdot \wedge \tau_M} - B^n_{\cdot \wedge \tau_M}\) is a local \(P^n\)-\(\F_+\)-martingale. Furthermore, by the definition of the stopping time \(\tau_M\) and the linear growth assumption (Condition \ref{cond: LG}), we see that \(X_{\cdot \wedge \tau_M} - B^n_{\cdot \wedge \tau_M}\) is \(P^n\)-a.s. bounded by a constant independent of \(n\), which, in particular, implies that it is a true \(P^n\)-\(\F_+\)-martingale. Now, it follows from \cite[Proposition~IX.1.4]{JS} that \(Y^{(1)}_{\cdot \wedge \tau_M \circ \Phi} - Y^{(2)}_{\cdot \wedge \tau_M \circ \Phi}\) is a \(Q\)-\(\F^*\)-martingale. Recalling that \(Y^{(2)}\) is \(Q\)-a.s. locally absolutely continuous by Step 2, this means that \(Y^{(1)}\) is a \(Q\)-\(\F^*\)-semimartingale with first characteristic \(Y^{(2)}\). Similarly, we see that the second characteristic is given by \(Y^{(3)}\). Finally, we need to relate these observations to the probability measure \(P\) and the filtration \(\F_+\). We denote by \(A^{p, \Phi^{-1}(\F_+)}\) the dual predictable projection of some process \(A\), defined on \((\Omega^*, \cF^*)\), to the filtration \(\Phi^{-1}(\F_+)\). Recall from \cite[Lemma 10.42]{jacod79} that, for every \(t \in \bR_+\), a random variable \(Z\) on \((\Omega^*, \cF^*)\) is \(\Phi^{-1}(\cF_{t+})\)-measurable if and only if it is \(\cF^*_t\)-measurable and \(Z (\omega^{(1)}, \omega^{(2)}, \omega^{(3)})\) does not depend on \((\omega^{(2)}, \omega^{(3)})\).
Thanks to Stricker's theorem (see, e.g., \cite[Lemma~2.7]{jacod80}), \(Y^{(1)}\) is a \(Q\)-\(\Phi^{-1} (\F_+)\)-semimartingale. 
Notice that each \(\tau_M \circ \Phi\) is a \(\Phi^{-1}(\F_+)\)-stopping time and recall from Step 3 that \((\llambda \otimes Q)\)-a.e. \((d Y^{(2)}/ d \llambda, d Y^{(3)}/ d \llambda) \in \Theta(Y^{(1)})\). Hence, by definition of \(\tau_M\) and the linear growth assumption, for every \(M \in D\), we have
\[
E^Q \big[ \on{Var} (Y^{(2)})_{\tau_M \circ \Phi} \big] + E^Q \big[ \on{Var}(Y^{(3)})_{\tau_M \circ \Phi} \big] = E^Q \Big[ \int_0^{\tau_M} \Big(\Big| \frac{d Y^{(2)}}{d \llambda} \Big| + \Big| \frac{d Y^{(3)}}{d \llambda} \Big| \Big) d \llambda \Big] < \infty,
\]
where \(\on{Var} (\cdot)\) denotes the variation process.
By virtue of this, we get from \cite[Proposition 9.24]{jacod79} that the \(Q\)-\(\Phi^{-1}(\F_+)\)-characteristics of \(Y^{(1)}\) are given by \(((Y^{(2)})^{p, \Phi^{-1}(\F_+)}, (Y^{(3)})^{p, \Phi^{-1}(\F_+)})\). 
Hence, thanks to Lemma~\ref{lem: jacod restatements} below, the coordinate process \(X\) is a \(P\)-\(\F_+\)-semimartingale whose characteristics \((B^P, C^P)\) satisfy \(Q\)-a.s.
\[(B^P, C^P) \circ \Phi = ((Y^{(2)})^{p, \Phi^{-1}(\F_+)}, (Y^{(3)})^{p, \Phi^{-1}(\F_+)}).\] Consequently, we deduce from the Steps~2 and 3, and \cite[Theorem~5.25]{HWY}, that \(P\)-a.s. \((B^P, C^P) \ll \llambda\) and 
\begin{align*}
(\llambda \otimes P) \big( (d B^P / d \llambda&, d C^P / d \llambda) \not \in \Theta(X) \big) 
\\&= (\llambda \otimes Q \circ \Phi^{-1}) \big( (d B^P / d \llambda, d C^P / d \llambda) \not \in \Theta(X) \big)
\\&= (\llambda \otimes Q) \big( E^Q [(d Y^{(2)} / d \llambda, d Y^{(3)} / d \llambda) | \Phi^{-1} (\F_+)_-] \not \in \Theta(Y^{(1)})\big) = 0,
\end{align*}
where we use \cite[Corollary 8, p. 48]{diestel} for the final equality.
This means that \(P \in \mathfrak{P}(\Theta) \) and the proof is complete.
\end{proof}

\begin{proposition} \label{prop: compactness}
Suppose that the Conditions \ref{cond: convexity}, \ref{cond: LG} and \ref{cond: continuity} hold.
For any compact set \(K \subset \mathbb{R}\), the set 
\begin{align*}
\cR^\circ := \big\{ P \in \fPas \colon P &\circ X_0^{-1} \in \{\delta_x \colon x \in K\}, \ (\llambda \otimes P)\text{-a.e. } (dB^{P} /d\llambda, dC^{P}/d\llambda) \in \Theta(X)  \big\}
\end{align*}
is compact in \(\mathfrak{P}(\Omega)\). 
\end{proposition}
\begin{proof}
Thanks to \cite[Lemma 7.4]{CN22}, we already know that \(\cR^\circ\) is relatively compact. We note that \(\cR^\circ \) is closed, being the intersection of the closed sets \( \mathfrak{P}(\Theta)\) and
\[ 
\big\{ P \in \mathfrak{P}(\Omega) \colon P \circ X_0^{-1} \in \{ \delta_x \colon x \in K \} \big \}.
\]
While \( \mathfrak{P}(\Theta)\) is closed by Proposition \ref{prop: closedness}, the latter set is closed as \(K\) is closed (\cite[Theorem 15.8]{hitchi}).
This completes the proof.
\end{proof}

\subsection{Upper Hemicontinuity of \(\mathcal{R}\)}

\begin{proposition} \label{prop: upper hemicontinuous}
Suppose that the Conditions \ref{cond: convexity}. \ref{cond: LG} and \ref{cond: continuity} hold. Then, the correspondence \( \cR \) is upper hemicontinuous.
\end{proposition}
\begin{proof}
This follows as a special case of Proposition \ref{prop: K upper hemi and compact} below.
\end{proof}

\subsection{Proof of Theorem \ref{thm: main r}}
By Proposition \ref{prop: compactness}, \( \cR \) is compact-valued, while Proposition \ref{prop: upper hemicontinuous} provides upper hemicontinuity of \( \cR \). \qed

\section{Proof of Theorem \ref{thm: new very main}} \label{sec: proof main result}
First of all, parts (iii) and (iv), i.e., the \(C_0\) and the (uniform) strong \(\usc_b\)--Feller properties, follow directly from the Feller selection principle given by Theorem \ref{theo: Feller selection}.

We now discuss part (i), i.e., the \(\usc_b\)--Feller property.
Let \( \psi \in \usc_b(\mathbb{R}; \mathbb{R})\).
Notice that 
\(
\of 0, \infty \of \hspace{0.1cm} \ni (t, \omega) \mapsto \psi(\omega(t))
\)
is upper semicontinuous and bounded. Thus, thanks to \cite[Theorem~8.10.61]{bogachev}, the map 
\begin{align} \label{eq: joint continuity semigroup proof}
\mathbb{R}_+ \times \mathfrak{P}(\Omega) \ni (t, P) \mapsto E^{P}\big[\psi(X_{ t}) \big]\end{align}
is upper semicontinuous, too.
By Theorem \ref{thm: main r}, the compact-valued correspondence \( \bR_+ \times \bR \ni (t,x) \mapsto \{t\} \times \cR(x) \) is upper hemicontinuous, being the finite product of upper hemicontinuous correspondences with compact values, cf. \cite[Theorem 17.28]{hitchi}. 
Thus, upper semicontinuity of 
\( (t,x) \mapsto \cE^x(\psi(X_t)) \)
follows from the upper semicontinuity of \eqref{eq: joint continuity semigroup proof} and
(a version of) Berge's maximum theorem as given by \cite[Lemma 17.30]{hitchi}. This completes the proof of (i).

Part (ii), i.e., the \(C_b\)--Feller property, follows along the same lines when additionally Theorem~\ref{theo: lower hemi} and \cite[Lemma~17.29]{hitchi} are taken into consideration. We omit the details for brevity.
\qed

\section{Markov and Feller Selection Principles: Proof of Theorems \ref{theo: strong Markov selection} and \ref{theo: Feller selection} } 
\label{sec: feller selection}

The proof of the strong Markov selection principle, given by Theorem \ref{theo: strong Markov selection}, is based on some fundamental ideas of Krylov \cite{krylov1973selection} for Markovian selection as worked out in the monograph \cite{SV} of Stroock and Varadhan, see also \cite{nicole1987compactification,hausmann86}. The main technical steps in the argument are to establish stability under conditioning and pasting of a certain sequence of correspondences. 

The proof of the Feller selection principle, given by Theorem \ref{theo: Feller selection}, is based on the observation that any strong Markov selection is already a (uniform) strong Feller and \(C_0\)--Feller selection in case the system carries enough randomness, which is ensured by our (uniform) ellipticity condition. 
\subsection{Proof of the Markov Selection Principle: Theorem \ref{theo: strong Markov selection}}
This section is split into two parts. We start with some properties of the correspondence \(\cK\) and then finalize the proof in the second part.
\subsubsection{Preparations}
The following lemma is a restatement of a path-continuous version of \cite[Lemma~III.3.38, Theorem~III.3.40]{JS}. 
\begin{lemma} \label{lem: JS convex}
Let \(P, Q \in \fPs := \fPs (0)\) and denote the characteristics of the coordinate process by \((B^P, C^P)\) and \((B^Q, C^Q)\), respectively. Further, take \(\alpha \in (0, 1)\) and set 
\[
R := \alpha P + (1 - \alpha) Q.
\]
Then, \(P \ll R, Q \ll R\) and there are versions of the Radon--Nikodym density processes \(dP/dR |_{\mathcal{F}_\cdot} = Z^P\) and \(dQ/dR |_{\mathcal{F}_\cdot} = Z^Q\) such that identically
\begin{align}\label{eq: ZP ZQ convex combi}
\alpha Z^P + (1 - \alpha) Z^Q = 1, \quad 0 \leq Z^P \leq 1/\alpha, \quad 0 \leq Z^Q \leq 1/(1 - \alpha).
\end{align}
Moreover, \(R \in \fPs\) and the \(R\)-characteristics \((B^R, C^R)\) of the coordinate process satisfy
\[
d B^R = \alpha Z^P d B^P + (1 - \alpha) Z^Q d B^Q, \qquad d C^R = \alpha Z^P d C^P + (1 - \alpha) Z^Q d C^Q.
\]
\end{lemma}

The following lemma is a restatement of \cite[Lemma 2.9 (a)]{jacod80} for a path-continuous setting. 
\begin{lemma} \label{lem: jacod restatements}
	Take two filtered probability spaces \(\B^* = (\Omega^*, \mathcal{F}^*, \F^* = (\mathcal{F}^*_t)_{t \geq 0}, P^*)\) and  \(\B' = (\Omega', \mathcal{F}', \F' = (\mathcal{F}'_t)_{t \geq 0}, P')\) with right-continuous filtrations and the property that there is a map \(\phi \colon \Omega' \to \Omega^*\) such that
	\(
	\phi^{-1} (\mathcal{F}^*) \subset \mathcal{F}',P^* = P' \circ \phi^{-1}\) and \(\phi^{-1} (\mathcal{F}^*_t) = \mathcal{F}'_t\) for all \(t \in \mathbb{R}_+\).
Then,
		 \(X^*\) is a continuous semimartingale on \(\B^*\) if and only if \(X' = X^* \circ \phi\) is a continuous semimartingale on \(\B'\). Moreover, \((B^*, C^*)\) are the characteristics of \(X^*\) if and only if \((B^* \circ \phi, C^* \circ \phi)\) are the characteristics of \(X' = X^* \circ \phi\).
\end{lemma}

For \(t \in \mathbb{R}_+\), we define \(\gamma_t \colon \Omega \to \Omega\) by \(\gamma_t (\omega) := \omega ( (\cdot - t)^+ ) \) for \(\omega \in \Omega\).
Moreover, for \(P \in \mathfrak{P}(\Omega)\) and \(t \in \mathbb{R}_+\), we set \[P_t := P \circ \theta_t^{-1}, \qquad P^t := P \circ \gamma_t^{-1}.\]

\begin{lemma} \label{lem: p^t cont}
    The maps \((t, P) \mapsto P_t\) and \((t, P) \mapsto P^t\) are continuous.
\end{lemma}
\begin{proof}
Notice that \((t, \omega) \mapsto \theta_t (\omega)\) and \((t, \omega) \mapsto \gamma_t (\omega)\) are continuous by the Arzel\`a--Ascoli theorem.
  Now, the claim follows from \cite[Theorem 8.10.61]{bogachev}. 
\end{proof}

\begin{lemma} \label{lem: implication c^*}
For every \((t, \omega) \in \of 0, \infty\of\), \(P \in \cC(t, \omega)\) implies \(P_t \in \cK (0, \omega(t))\).
\end{lemma}
\begin{proof}
Let \((t, \omega) \in \of 0, \infty\of\) and take \(P \in \cC(t, \omega)\). Obviously, \( P_t \circ X_0^{-1} = \delta_{\omega(t)} \) and, thanks to
Lemma~\ref{lem: jacod restatements}, we also get \(P_t \in \fPas \) and
\(
(\llambda \otimes P_t)\text{-a.e. } (dB^{P_t} /d\llambda, dC^{P_t}/d\llambda) \in \Theta(X),
\)
which proves \( P_t \in \cK(0,\omega(t)) \).
\end{proof}

\begin{lemma} \label{lem: p^t}
For every \( (t, x) \in \bR_+ \times \bR\), we have
\( P \in \cK(0,x) \) if and only if \( P^t \in \cK(t,x) \).
\end{lemma}
\begin{proof}
Let \(x \in \bR\) and \( P \in \cK(0,x) \). As 
\(
\gamma_t^{-1} (\{ X = x \text{ on } [0,t] \} ) = \{ X_0 = x \},
\)
we have \( P^t( X = x \text{ on } [0,t]) = 1\).
Next, it follows from
Lemma \ref{lem: jacod restatements} that \(P^t \in \fPas (t) \) and
\[
(\llambda \otimes P^t)\text{-a.e. } (dB^{P^t}_{\cdot + t} /d\llambda, dC^{P^t}_{\cdot + t}/d\llambda) \in  \Theta (X_{\cdot + t}),
\]
which proves \( P^t \in \cK(t,x) \).
Conversely, take \(P^t \in \cK(t,x) \). Due to the identity \( \theta_t \circ \gamma_t = \on{id} \), we have
\( P =  (P^t)_t \). Thus, \( P \circ X_0 = \delta_x \), and applying Lemma \ref{lem: jacod restatements} once more, we conclude that~\(P \in \cK(0,x) \). The proof is complete.
\end{proof}

\begin{lemma} \label{lem: k idenity}
    For all \((t, x) \in \bR_+ \times \bR\), we have \(\cK (t, x) = \{P^t \colon P \in \cK(0, x)\}\).
\end{lemma}
\begin{proof}
Take \((t, x) \in \bR_+ \times \bR\). Lemma \ref{lem: p^t} yields the inclusion \(\{P^t \colon P \in \cK(0, x)\} \subset \cK(t, x)\). Conversely, take \(P \in \cK(t, x)\). As 
\( P( X = x \text{ on } [0,t]) = 1 \), the equality
\( (P_t)^t  = P \) holds. Now, Lemma~\ref{lem: implication c^*} yields that \(P_t \in \cK(0, x)\) and hence, we get the inclusion \(\cK(t, x) \subset \{P^t \colon P \in \cK(0, x)\}\). 
\end{proof}

\begin{proposition} \label{prop: K upper hemi and compact}
The correpondence \((t, x) \mapsto \cK(t, x)\) is upper hemicontinuous with nonempty and compact values.
\end{proposition}
\begin{proof}
As \( x \mapsto \cK(0, x) \) has nonempty compact values by Proposition \ref{prop: compactness} and Standing Assumption~\ref{SA: non empty}, Lemmata~\ref{lem: p^t cont} and \ref{lem: k idenity} yield that the same is true for \((t, x) \mapsto \cK (t, x)\).

It remains to show that \(\cK\) is upper hemicontinuous. 
Let \( F \subset \mathfrak{P}(\Omega) \) be closed. We need to show that
\( \cK^l(F) = \{ (t, x) \in \bR_+ \times \bR \colon \cK(t, x) \cap F \neq \emptyset \} \) is closed.
Suppose that the sequence \( (t^n, x^n)_{n \in \mathbb{N}} \subset \cK^l(F) \) converges to \((t, x) \in \bR_+ \times \bR\).
For each \(n \in \mathbb{N} \), there exists a probability measure \( P^n \in \cK (t^n, x^n) \cap F \).
Thanks to Proposition \ref{prop: compactness}, the set
\begin{align*}
\cR^\circ := \big\{ P \in \fPas \colon P &\circ X_0^{-1} \in \{\delta_{x^n}, \delta_{x} \colon n \in \mathbb{N}\},\ (\llambda \otimes P)\text{-a.e. } (dB^{P} /d\llambda, dC^{P}/d\llambda) \in \Theta(X)  \big\}
\end{align*}
is compact. Hence, by Lemma \ref{lem: p^t cont}, so is the set 
\[
\cK^\circ := \{ P^t \colon (t, P) \in \{t^n, t \colon n \in \mathbb{N}\} \times \cR^\circ \}.
\]
Thus, by virtue of Lemma \ref{lem: k idenity}, we conclude that \(\{P^n \colon n \in \mathbb{N}\} \subset \cK^\circ\) is relatively compact. 
Hence, passing to a subsequence if necessary, we can assume that \( P^n \to P \) weakly for some 
\( P \in \cK^\circ \cap F\).
Notice that, for every \(\varepsilon \in (0, t)\), the set \(\{|X_s - x| \leq \varepsilon \text{ for all } s \in [0, t - \varepsilon]\} \subset \Omega\) is closed. Consequently, by the Portmanteau theorem, for every \(\varepsilon \in (0, t)\), we get
\[
1 = \limsup_{n \to \infty} P^n( |X_s - x| \leq \varepsilon \text{ for all } s \in [0, t - \varepsilon] ) \leq P( |X_s - x| \leq \varepsilon \text{ for all } s \in [0, t - \varepsilon] ).
\]
Consequently, \(P( X = x \text{ on } [0, t] ) = 1\), which implies that \(P = (P_t)^t\).
By Lemmata \ref{lem: p^t cont} and \ref{lem: implication c^*}, we have \((P^n)_{t^n} \in \cK(0, x^n)\) and \((P^n)_{t^n} \to P_t\) weakly. Further, since \(P_t \circ X_0^{-1} = \delta_x\), Proposition~\ref{prop: closedness} yields that \(P_t \in \cK(0, x)\). 
Thus, by Lemma \ref{lem: p^t}, \(P \in \cK^\circ \cap F \cap \cK(t, x) = \cK(t, x) \cap F,\) which implies
\((t, x) \in \cK^l(F) \). We conclude that \(\cK\) is upper hemicontinuous.
\end{proof}

\begin{lemma} \label{lem: r^* compact}
The correspondence \( (t, x) \mapsto \cK(t, x) \) has convex values.
\end{lemma}
\begin{proof}
By virtue of Lemma \ref{lem: k idenity}, it suffices to prove that \(\cK (0, x)\) is convex for every fixed \(x \in \mathbb{R}\).
Indeed, for every \(P, Q \in \cK(t, x)\) and \(\alpha \in (0, 1)\), there are probability measures \(\oP, \oQ \in \cK(0, x)\) such that \(\oP^t = P\) and \(\oQ^t = Q\). Then, \(\alpha P + (1 - \alpha) Q = (\alpha \oP + (1 - \alpha) \oQ)^t\) and consequently, from Lemma~\ref{lem: p^t}, we get \(\alpha P + (1 - \alpha) Q \in \cK(t, x)\) once \(\alpha \oP + (1 - \alpha) \oQ \in \cK(0, x)\).

We now prove the convexity of \(\cK(0, x)\).
Take \(P, Q \in \cK (0, x)\) and \(\alpha \in (0, 1)\). Furthermore, set \(R := \alpha P + (1 - \alpha) Q\). It is easy to see that
\(
R \circ X^{-1}_0 = \delta_x.
\)
By Lemma~\ref{lem: JS convex}, using also its notation, \(R \in \fPas\) and the Lebesgue densities \((b^R, a^R)\) of the \(R\)-characteristics of the coordinate process are given by 
\[
b^R = \alpha Z^P b^P + (1 - \alpha) Z^Q b^Q, \qquad a^R = \alpha Z^P a^P + (1 - \alpha) Z^Q a^Q.
\]
Since \(P \in \cK (0, x)\), we have 
\[
\iint Z^P \1_{\{(b^P, a^P) \hspace{0.05cm}\not \in \hspace{0.05cm} \Theta(X), \hspace{0.05cm} Z^P \hspace{0.05cm} >\hspace{0.05cm} 0\}} d (\llambda \otimes R) 
= (\llambda \otimes P) ( (b^P, a^P) \not \in \Theta(X), Z^P > 0) = 0.
\]
Thus, \((\llambda \otimes R)\)-a.e. \[\1_{\{ (b^P, a^P) \hspace{0.05cm}\not \in \hspace{0.05cm} \Theta(X), \hspace{0.05cm} Z^P > 0\}} = 0.\] Similarly, we obtain that \((\llambda \otimes R)\)-a.e. \[\1_{\{ (b^Q, a^Q) \hspace{0.05cm}\not \in \hspace{0.05cm} \Theta(X), \hspace{0.05cm} Z^Q > 0\}} = 0.\]
Consequently, recalling that \(\{Z^P = 0, Z^Q = 0\} = \emptyset\), by virtue of \eqref{eq: ZP ZQ convex combi}, and using that \(\Theta\) is convex-valued (Condition~\ref{cond: convexity}), we get 
\begin{align*}
    (\llambda \otimes R) &\big( (b^R, a^R) \not \in \Theta(X) \big) 
    \\&=(\llambda \otimes R) \big( (b^R, a^R) \not \in \Theta(X), (b^P, a^P) \in \Theta(X), Z^P > 0, (b^Q, a^Q) \in \Theta(X), Z^Q > 0 \big)
    \\& \qquad \qquad + (\llambda \otimes R) \big( (b^P, a^P) \not \in \Theta(X), (b^P, a^P) \in \Theta(X), Z^P > 0, Z^Q = 0 \big) 
    \\& \qquad \qquad + (\llambda \otimes R) \big( (b^Q, b^Q) \not \in \Theta(X), Z^P = 0, (b^Q, a^Q) \in \Theta(X), Z^Q > 0 \big)
    \\& = 0.
\end{align*}
We conclude that \(R \in \cK(0, x)\). The proof is complete.
\end{proof}

\begin{lemma} \label{lem: iwie Markov}
    Let \(Q \in \mathfrak{P}(\Omega)\) and take \(t \in \bR_+\) and \(\omega, \alpha \in \Omega\) such that \(\omega (t) = \alpha (t)\).    Then, 
    \[
    \delta_\alpha \otimes_t Q \in \cC(t, \alpha) \quad \Longleftrightarrow \quad \delta_\omega \otimes_t Q \in \cC (t, \omega).
    \]
\end{lemma}
\begin{proof}
Set \(\oQ := \delta_\alpha \otimes_t Q\) and \(\oP := \delta_\omega \otimes_t Q\).
Suppose that \( \oQ \in \cC(t,\alpha) \).
Thanks to Lemma \ref{lem: implication c^*}, we have \(\oQ_t \in \cC(0, \omega (t))\). Since \(\oQ_t = \oP_t\), we also have \(\oP_t \in \cC(0, \omega (t))\). 
Thus, Lemma \ref{lem: jacod restatements} yields that \(\oP \in \fPas (t)\) and \((\oP \otimes \llambda)\)-a.e.
\[
(d B^{\oP}_{\cdot + t}/d \llambda, d C^{\oP}_{\cdot + t}/ d \llambda) \in \Theta ( X \circ \theta_t ) = \Theta (X_{\cdot + t}),
\]
which implies \( \oP \in \cC(t, \omega) \). 
The converse implication follows by symmetry.
\end{proof}

\begin{definition}
	A correspondence \(\cU \colon \bR_+ \times \bR \twoheadrightarrow \mathfrak{P}(\Omega)\) is said to be 
	\begin{enumerate}
		\item[\textup{(i)}]
		\emph{stable under conditioning} if for any \((t, x) \in \bR_+ \times \bR\), any stopping time \(\tau\) with \(t \leq \tau < \infty\), and any \(P \in \cU(t, x)\), there exists a \(P\)-null set \(N \in \cF_\tau\) such that 
		\(\delta_{\omega (\tau(\omega))} \otimes_{\tau (\omega)} P (\cdot | \mathcal{F}_\tau) (\omega) \in \cU(\tau (\omega), \omega (\tau (\omega)))\) for all \(\omega \not \in N\);
		\item[\textup{(ii)}]
		\emph{stable under pasting} if for any \((t, x) \in \bR_+ \times \bR\), any stopping time \(\tau\) with \(t \leq \tau < \infty\), any \(P \in \cU(t, x)\) and any \(\mathcal{F}_\tau\)-measurable map \(\Omega \ni \omega \mapsto Q_\omega \in \mathfrak{P}(\Omega)\) the following implication holds:
		\[
		P\text{-a.a. } \omega \in \Omega \quad \delta_{\omega (\tau(\omega))} \otimes_{\tau (\omega)} Q_\omega \in \cU (\tau (\omega), \omega (\tau(\omega)))\quad \Longrightarrow \quad P \otimes_\tau Q \in \cU(t, x).
		\]
	\end{enumerate}
\end{definition}

\begin{lemma} \label{lem: K stable under both}
	The correspondence \(\cK\) is stable under conditioning and pasting.
\end{lemma}
\begin{proof}
	Stability under conditioning follows from \cite[Corollary~6.12]{CN22} and Lemma~\ref{lem: iwie Markov}, and stability under pasting follows from Lemma \ref{lem: iwie Markov} and \cite[Lemma~6.17]{CN22}.
\end{proof}

Recall from \cite[Definition 18.1]{hitchi} that a correspondence \(\mathcal{U} \colon \bR_+ \times \bR \twoheadrightarrow \mathfrak{P}(\Omega)\) is called \emph{measurable} if the lower inverse \(\{ (t, x) \in \bR_+ \times \bR \colon \mathcal{U} (t, x) \cap F \not = \emptyset\}\) is Borel for every closed set \(F \subset \mathfrak{P}(\Omega)\).

\begin{lemma} \label{lem: U to U*}
	Suppose that \(\cU \colon \bR_+ \times \bR \twoheadrightarrow \mathfrak{P}(\Omega)\) is a measurable correspondence with nonempty and compact values such that, for all \((t, x) \in \bR_+ \times \bR\) and \(P \in \cU (t, x)\), \(P (X_s = x\text{ for all } s \in [0, t])= 1\). Suppose further that \(\cU\) is stable under conditioning and pasting. Then, for any \(\phi \in \usc_b (\bR; \bR)\), the correspondence 
	\[
	\cU^* (t, x) := \Big\{ P \in \cU (t, x) \colon E^P \big[ \phi (X_T) \big] = \sup_{Q \in \cU (t, x)} E^Q \big[ \phi (X_T) \big] \Big\}
	\]
	is also measurable with nonempty and compact values and it is stable under conditioning and pasting. Further, if \(\cU\) has convex values, then so does \(\cU^*\).
\end{lemma}
\begin{proof}
	We adapt the proof of \cite[Lemma 12.2.2]{SV}, see also the proof of \cite[Lemma~3.4 (a, d)]{hausmann86}.	
 As \(\psi\) is assumed to be upper semicontinuous, \cite[Theorem 2.43]{hitchi} implies that \( \cU^* \) has nonempty and compact values. Moreover, \cite[Theorem~18.10]{hitchi} and \cite[Lemma~12.1.7]{SV} imply that \(\cU^*\) is measurable.
	The final claim for the convexity is obvious. It is left to show that \(\cU^*\) is stable under conditioning and pasting.
	Take \((t, x) \in \bR_+ \times \bR, P \in \cU^* (t, x)\) and let \(\tau\) be a stopping time such that \(t \leq \tau < \infty\). We define 
	\begin{align*}
	N &:= \big\{ \omega \in \Omega \colon P_\omega := \delta_{\omega (\tau (\omega))} \otimes_{\tau (\omega)} P (\cdot | \cF_\tau) (\omega) \not \in \cU (\tau (\omega), \omega (\tau (\omega)))\big\}, \\
	A &:= \big\{ \omega \in \Omega \backslash N \colon P_\omega \not \in \cU^* (\tau (\omega), \omega (\tau (\omega)))\big\}.
	\end{align*}
	As \(\cU\) is stable under conditioning, we have \(P(N) = 0\). By \cite[Lemma 12.1.9]{SV}, \(N, A \in \cF_\tau\). As we already know that \(\cU^*\) is measurable, by virtue of \cite[Theorem 12.1.10]{SV}, there exists a measurable map \((s, y) \mapsto R (s, y)\) such that \(R (s, y) \in \cU^* (s, y)\). We set \(R_\omega := R (\tau (\omega), \omega (\tau (\omega)))\), for \(\omega \in \Omega\), and note that \(\omega \mapsto R_\omega\) is \(\cF_\tau\)-measurable. Further, we set 
	\[
	Q_\omega := \begin{cases} R_\omega, & \omega \in N \cup A,\\
	P_\omega, & \omega \not \in N \cup A. \end{cases}
	\]
	By definition of \(R\) and \(N\), \(Q_\omega \in \cU (\tau(\omega), \omega (\tau (\omega)))\) for all \(\omega \in \Omega\).
	As \(\cU\) is stable under pasting, we have \(P \otimes_{\tau} Q\in \cU(t, x)\) and we obtain
	\begin{align*}
	\sup_{Q^* \in \cU (t, x)} &E^{Q^*} \big[ \phi (X_T) \big] 
	\\&\geq E^{P \otimes_\tau Q} \big[ \phi (X_T) \big]
	\\&= \int_{N\cup A} E^{\delta_\omega \otimes_{\tau (\omega)} R_\omega} \big[ \phi (X_T) \big] P (d \omega) + E^P \big[ \1_{N^c \cap A^c}E^P \big[ \phi (X_T) | \cF_\tau \big] \big] 
	\\&= \int_{A} \big[ E^{\delta_\omega \otimes_{\tau (\omega)} R_\omega} \big[ \phi (X_T) \big] - E^{\delta_\omega \otimes_{\tau (\omega)} P_\omega} \big[ \phi (X_T) \big] \big] P (d \omega) + \sup_{Q^* \in \cU (t, x)} E^{Q^*} \big[ \phi (X_T) \big]
	\\&= \int_{A} \big[ E^{R_\omega} \big[ \phi (X_T) \big] - E^{P_\omega} \big[ \phi (X_T) \big] \big] P (d \omega) + \sup_{Q^* \in \cU (t, x)} E^{Q^*} \big[ \phi (X_T) \big].
	\end{align*}
As \(E^{R_\omega} \big[ \phi (X_T) \big] > E^{P_\omega} \big[ \phi (X_T) \big]\) for all \(\omega \in A\), we conclude that \(P (A) = 0\). This proves that \(\cU^*\) is stable under conditioning. 

Next, we prove stability under pasting. Let \((t, x) \in \bR_+ \times \bR\), take a stopping time \(\tau\) with \(t \leq \tau < \infty\), a probability measure \(P \in \cU^*(t, x)\) and an \(\mathcal{F}_\tau\)-measurable map \(\Omega \ni \omega \mapsto Q_\omega \in \mathfrak{P}(\Omega)\) such that, for \(P\)-a.a. \(\omega \in \Omega\), 
\(\delta_{\omega (\tau(\omega))} \otimes_{\tau (\omega)} Q_\omega \in \cU^* (\tau (\omega), \omega (\tau (\omega)))\). 
As \(\cU\) is stable under pasting, we have \(P \otimes_\tau Q \in \cU (t, x)\). Further, recall that \(\delta_{\omega (\tau (\omega))} \otimes_{\tau (\omega)} P (\cdot | \cF_\tau) (\omega)\in \cU (\tau (\omega), \omega (\tau (\omega)))\) for \(P\)-a.a. \(\omega \in \Omega\), as \(\cU\) is stable under conditioning. Thus, we get 
	\begin{align*}
\sup_{Q^* \in \cU (t, x)} E^{Q^*} \big[ \phi (X_T) \big] &\geq E^{P \otimes_\tau Q} \big[ \phi (X_T) \big] 
\\&= \int E^{\delta_\omega \otimes_{\tau (\omega)} Q_\omega} \big[ \phi (X_T) \big] P (d \omega)
\\&= \int E^{\delta_{\omega (\tau (\omega))} \otimes_{\tau (\omega)} Q_\omega} \big[ \phi (X_T) \big] \1_{\{\tau (\omega) < T\}}P (d \omega) + E^P \big[ \phi (X_T) \1_{\{T \leq \tau\}}\big]
\\&= \int \sup_{Q^* \in \cU (\tau (\omega), \omega (\tau (\omega)))} E^{Q^*} \big[ \phi (X_T) \big] \1_{\{\tau (\omega) < T\}} P (d \omega) + E^P \big[ \phi (X_T) \1_{\{T \leq \tau\}}\big]
\\&\geq \int E^{\delta_{\omega (\tau (\omega))} \otimes_{\tau (\omega)} P (\cdot | \cF_\tau)(\omega)} \big[ \phi (X_T) \big] \1_{\{\tau (\omega) < T\}} P (d \omega) + E^P \big[ \phi (X_T) \1_{\{T \leq \tau\}}\big]
\\&= E^P \big[ E^P \big[ \phi (X_T) | \cF_\tau \big] \1_{\{\tau < T\}} \big] + E^P \big[ \phi (X_T) \1_{\{T \leq \tau\}}\big]
\\&= E^P \big[ \phi (X_T) \big]
\\&= \sup_{Q^* \in \cU (t, x)} E^{Q^*} \big[ \phi (X_T) \big].
\end{align*}
This implies that \(P \otimes_\tau Q\in \cU^* (t, x)\). The proof is complete.
\end{proof}

\subsubsection{Proof of Theorem \ref{theo: strong Markov selection}}
	We adapt the proofs of \cite[Theorems 6.2.3 and 12.2.3]{SV}, cf. also the proofs of \cite[Proposition 6.6]{nicole1987compactification} and \cite[Proposition 3.2]{hausmann86}.
	
Fix a finite time horizon \(T > 0\) and a function \(\psi \in \usc_b(\bR; \bR)\).
Let \(\{\sigma_n \colon n \in \mathbb{N}\}\) be a dense subset of \((0, \infty)\) and let \(\{\phi_n \colon n \in \mathbb{N}\}\) be a dense subset of \(C_c (\bR)\). Furthermore, let \((\lambda_N, f_N)_{N \in \mathbb{N}}\) be an enumeration of \(\{(\sigma_m, \phi_n) \colon n, m \in \mathbb{N}\}\). 
For \((t, x) \in \bR_+ \times \bR\), define inductively  
\[
\cK^*_0 (t, x) := \Big\{ P \in \cK (t, x)\colon E^P \big[ \psi (X_T) \big] = \sup_{Q \in \cK(t, x)} E^Q \big[ \psi (X_T) \big] \Big\}
\]
and 
\[
\cK^*_{N + 1} (t, x) := \Big\{ P \in \cK^*_N (t, x) \colon E^P \big[ f_{N + 1} (X_{\lambda_{N + 1}}) \big] = \sup_{Q \in \cK^*_N (t, x)} E^Q \big[ f_{N + 1} (X_{\lambda_{N + 1}}) \big] \Big\}, \quad N \in \mathbb{Z}_+.
\]
Moreover, we set \[\cK^*_\infty (t, x) := \bigcap_{N = 0}^\infty \cK_N (t, x).\]

Thanks to Proposition \ref{prop: K upper hemi and compact} and Lemmata \ref{lem: r^* compact} and \ref{lem: K stable under both}, the correspondence \(\cK\) is measurable with nonempty convex and compact values and it is further stable under conditioning and pasting. Thus, by Lemma \ref{lem: U to U*}, the same is true for \(\cK^*_0\) and, by induction, also for every \(\cK^*_N, N \in \mathbb{N}\). 
As (arbitrary) intersections of convex and compact sets are itself convex and compact, \(\cK^*_\infty\) has convex and compact values. Further, by Cantor's intersection theorem, \(\cK^*_\infty\) has nonempty values, and, by \cite[Lemma~18.4]{hitchi}, \(\cK^*_\infty\) is measurable. Moreover, it is clear that \(\cK^*_\infty\) is stable under conditioning, as this is the case for every \(\cK^*_N, N \in \mathbb{Z}_+\). 

We now show that \(\cK^*_\infty\) is singleton-valued. Take \(P, Q \in \cK^*_\infty (t, x)\) for some \((t, x) \in \bR_+ \times \bR\). By definition of \(\cK^*_\infty\), we have 
\[
E^P \big[ f_N (X_{\lambda_N}) \big] = E^Q \big[ f_N (X_{\lambda_N})\big], \quad N \in \mathbb{N}.
\]
This implies that \(P \circ X_s^{-1} = Q \circ X_s^{-1}\) for all \(s \in \bR_+\). Next, we prove that 
\[
E^P \Big[ \prod_{k = 1}^n g_k (X_{t_k}) \Big] = E^Q \Big[ \prod_{k = 1}^n g_k (X_{t_k}) \Big] 
\]
for all \(g_1, \dots, g_n \in C_b (\bR; \bR), t \leq t_1 < t_2 < \dots < t_n < \infty\) and \(n \in \mathbb{N}\). We use induction over \(n\). For \(n = 1\) the claim is implied by the equality \(P \circ X_s^{-1} = Q \circ X_s^{-1}\) for all \(s \in \bR_+\). Suppose that the claim holds for \(n \in \mathbb{N}\) and take test functions \(g_1, \dots, g_{n + 1} \in C_b (\bR; \bR)\) and times \(t \leq t_1 < \dots < t_{n + 1} < \infty\).  We define 
\[
\mathcal{G}_n := \sigma (X_{t_k}, k = 1, \dots, n).
\]
Since 
\[
E^P \Big[ \prod_{k = 1}^{n + 1} g_k (X_{t_k}) \Big] = E^P \Big[ E^P \big[ g_{n + 1} (X_{t_{n + 1}}) | \mathcal{G}_n \big] \prod_{k = 1}^n g_k (X_{t_k}) \Big], 
\]
it suffices to show that \(P\)-a.s. 
\[
E^P \big[ g_{n + 1} (X_{t_{n + 1}}) | \mathcal{G}_n \big] = E^Q \big[ g_{n + 1} (X_{t_{n + 1}}) | \mathcal{G}_n \big].
\]
As \(\cK^*_\infty\) is stable under conditioning, there exists a null set \(N_1 \in \cF_{t_n}\) such that \(\delta_{\omega (t_n)} \otimes_{t_n} P (\cdot | \cF_{t_n}) (\omega) \in \cK^*_\infty (t_n, \omega (t_n))\) for all \(\omega \not \in N_1\). Notice that, by the tower rule, there exists a \(P\)-null set \(N_2 \in \mathcal{G}_n\) such that, for all \(\omega \not \in N_2\) and all \(A \in \cF\), 
\begin{equation} \label{eq: condi}
\begin{split}
\int \delta_{\omega' (t_n)} \otimes_{t_n} P (A | \cF_{t_n}) (\omega') P (d \omega' | \mathcal{G}_n) (\omega) &= \iint \1_A (\omega' (t_n) \otimes_{t_n} \alpha) P (d \alpha | \cF_{t_n}) (\omega' ) P (d \omega' | \mathcal{G}_n) (\omega) 
\\&= \iint \1_A (\omega (t_n) \otimes_{t_n} \alpha) P (d \alpha | \cF_{t_n}) (\omega' ) P (d \omega' | \mathcal{G}_n) (\omega) 
\\&=  \int \1_A (\omega (t_n) \otimes_{t_n} \omega') P (d \omega' | \mathcal{G}_n) (\omega) 
\\&= (\delta_{\omega (t_n)} \otimes_{t_n} P (\cdot | \mathcal{G}_n) (\omega)) (A).
\end{split}
\end{equation}
Let \(N_3 := \{P (N_1 | \mathcal{G}_n) > 0\} \in \mathcal{G}_n\). Clearly, \(E^P [ P (N_1| \mathcal{G}_n)] = P(N_1) = 0\), which implies that \(P (N_3) = 0\). Take \(\omega \not \in N_2 \cup N_3\). As \(\cK^*_\infty\) has convex and compact values and \(\delta_{\omega' (t_n)} \otimes_{t_n} P (\cdot | \cF_{t_n}) (\omega') \in \cK^*_\infty (t_n, \omega' (t_n))\) for all \(\omega' \not \in N_1\), we have  
\[
\int \delta_{\omega' (t_n)} \otimes_{t_n} P (A | \cF_{t_n}) (\omega') P (d \omega' | \mathcal{G}_n) (\omega) \in \cK^*_\infty (t_n, \omega (t_n)).
\]
Consequently, by virtue of \eqref{eq: condi}, we conclude that \(\delta_{\omega (t_n)} \otimes_{t_n} P (\cdot | \mathcal{G}_n) (\omega) \in \cK^*_\infty (t_n, \omega (t_n))\). Similarly, there exists a \(Q\)-null set \(N_4 \in \mathcal{G}_n\) such that \(\delta_{\omega (t_n)} \otimes_{t_n} Q (\cdot | \mathcal{G}_n) (\omega) \in \cK^*_\infty (t_n, \omega (t_n))\) for all \(\omega \not \in N_4\). Set \(N := N_2 \cup N_3 \cup N_4\). As \(P = Q\) on \(\mathcal{G}_n\), we get that \(P (N) = 0\). For all \(\omega \not \in N\), the induction base implies that 
\begin{align*}
E^P\big[ g_{n + 1} (X_{t_{n + 1}}) | \mathcal{G}_n\big] (\omega) &= E^{\delta_{\omega (t_n)} \otimes_{t_n} P (\cdot | \mathcal{G}_n)(\omega)} \big[ g_{n + 1} (X_{t_{n + 1}})\big] 
\\&= E^{\delta_{\omega (t_n)} \otimes_{t_n} Q (\cdot | \mathcal{G}_n)(\omega)} \big[ g_{n + 1} (X_{t_{n + 1}})\big]
\\&= E^Q\big[ g_{n + 1} (X_{t_{n + 1}}) | \mathcal{G}_n\big] (\omega).
\end{align*}
The induction step is complete and hence, \(P = Q\). 

We proved that \(\cK^*_\infty\) is singleton-valued and we write \(\cK^*_\infty (s, y) = \{P_{(s, y)}\}\). By the measurability of \(\cK^*_\infty\), the map \((s, y) \mapsto P_{(s,y)}\) is measurable. It remains to show the strong Markov property of the family \(\{P_{(s, y)} \colon (s, y) \in \bR_+ \times \bR\}\). Take \((s, y) \in \bR_+ \times \bR\). As \(\cK^*_\infty\) is stable under conditioning, for every finite stopping time \(\tau \geq s\), there exists a \(P_{(s, x)}\)-null set \(N\) such that, for all \(\omega \not \in N\),
\[
\delta_{\omega (\tau (\omega))} \otimes_{\tau (\omega)} P_{(s, y)} (\cdot | \cF_{\tau})(\omega) \in \cU_\infty (\tau (\omega), \omega (\tau (\omega))) = \{P_{(\tau (\omega), \omega (\tau (\omega)))}\}.\]
This yields, for all \(\omega \not \in N\), that 
\[
P_{(s, y)} (\cdot | \cF_\tau)(\omega) = \delta_\omega \otimes_{\tau (\omega)} \big[ \delta_{\omega (\tau (\omega))} \otimes_{\tau (\omega)} P_{(s, y)} (\cdot | \cF_{\tau})(\omega)\big] = \delta_\omega \otimes_{\tau (\omega)} P_{(\tau (\omega), \omega (\tau (\omega)))}.
\]
This is the strong Markov property and consequently, the proof is complete.
\qed

\begin{remark}
	Notice that the strong Markov property of the selection \(\{P_{(s, y)} \colon (s, y) \in \bR_+ \times \bR\}\) follows solely from the stability under conditioning property of \(\cK^*_\infty\). We emphasise that the stability under pasting property of each \(\cK^*_N, N \in \mathbb{Z}_+,\) is crucial for its proof. Indeed, in Lemma \ref{lem: U to U*}, the fact that \(\cU\) is stable under pasting has been used to establish that \(\cU^*\) is stable under conditioning.
\end{remark}

\subsection{Proof of the Strong Feller Selection Principle: Theorem \ref{theo: Feller selection}}
We start with the following partial extension of Lemma \ref{lem: maximal inequality}.
\begin{lemma}\label{lem: moment bound abvanced}
	Suppose that Condition \ref{cond: LG} holds. Let \(T, m > 0\) and let \(K \subset \bR\) be bounded. Then, 
	\[
	\sup_{s \in [0, T]} \sup_{x \in K} \sup_{P \in \cK (s, x)} E^P \Big[ \sup_{r \in [0, T]} | X_r |^m \Big] < \infty.
	\]
\end{lemma}
\begin{proof}
	For every \((s, x) \in [0, T] \times \bR\), Lemma \ref{lem: k idenity} yields that
	\[
	 \sup_{P \in \cK (s, x)} E^P \Big[ \sup_{r \in [0, T]} | X_r |^m \Big] =  \sup_{P \in \cK (0, x)} E^{P^s}\Big[ \sup_{r \in [0, T]} | X_r |^m \Big] \leq  \sup_{P \in \cK (0, x)} E^P \Big[ \sup_{r \in [0, T]} | X_r |^m \Big].
	\]
	Now, the claim follows from Lemma \ref{lem: maximal inequality}.
\end{proof}

\begin{theorem} \label{theo: selection is Feller}
	Suppose that the Conditions \ref{cond: LG}, \ref{cond: continuity}  and \ref{cond: ellipticity} hold. Let \(\p := \{P_{ (t, x)} \colon (t, x) \in \bR_+ \times \bR\}\) be a strong Markov family such that \(P_{(t, x)} \in \cK (t, x)\) for all \((t, x) \in \bR_+ \times \bR\). Then, \(\p\) has the strong Feller and the \(C_0\)--Feller property. 
    If in addition the Conditions \ref{cond: bdd} and \ref{cond: uniform ellipticity} hold, then \(\p\) has the uniform strong Feller property.
\end{theorem}
\begin{proof}
	First of all, thanks to the fundamental results \cite[Theorems 6.24, 7.14 (iii) and 7.16 (i)]{cinlar80} about Markovian It\^o semimartingales, there are two Borel functions \(\mu \colon \bR_+ \times \bR \to \bR\) and \(\sigma^2 \colon \bR_+ \times \bR \to \bR_+\) such that, for every \((s, x) \in \bR_+ \times \bR\), \(P_{(s, x)}\)-a.s. for \(\llambda\)-a.a. \(t \in \bR_+\)
	\[
	b^P_{t + s} = \mu (t + s, X_t), \quad a^P_{t + s} = \sigma^2 (t + s, X_t).
	\]
	By virtue of the Conditions \ref{cond: LG}, \ref{cond: continuity}  and \ref{cond: ellipticity}, we can w.l.o.g. assume that \(b\) and \(\sigma^2\) are locally bounded and that \(\sigma^2\) is locally bounded away from zero. 
	For \(M > 0\), we set 
		\begin{align*}
	\mu_M (t, x) &:= \begin{cases}
	\mu (t, x), & (t, x) \in [0, M] \times [-M, M], \\
	0, & \text{otherwise}, 
	\end{cases}
\\
	\sigma^2_M (t, x) &:= \begin{cases}
	\sigma^2 (t, x), & (t, x) \in [0, M] \times [-M, M], \\
	\sigma^2 (t \wedge M, x_0), & \text{otherwise},
	\end{cases}
	\end{align*}
	where \(x_0 \in \bR\) is an arbitrary, but fixed, reference point. Furthermore, we define 
	\[
	\rho_M^s := \inf \{t \geq s \colon |X_t| \geq M\} \wedge M, \quad s \in \bR_+, M > 0.
	\]
	Recall from \cite{SV} that a probability measure \(P\) on \((\Omega, \cF)\) is said to be a \emph{solution to the martingale problem for \((\mu_M, \sigma^2_M)\) starting from \((s, x) \in \bR_+ \times \bR\)} if \(P( X_t = x \text{ for all } t \in [0, s]) = 1\) and the processes
		\[
		f (X_t) - \int_s^t \big[ \mu (r, X_r) f' (X_r) + \tfrac{1}{2} \sigma^2 (r, X_r) f'' (X_r) \big] dr \colon \ t \geq s, \ \ f \in C^\infty_c (\bR; \bR), 
		\]
		are \(P\)-martingales.
	Due to \cite[Exercise 7.3.3]{SV} (see also \cite[Theorem 7.1.6]{SV}), for every starting value \((s, x) \in \bR_+ \times \bR\), there exists a unique solution \(P^M_{(s, x)}\) to the martingale problem for \((\mu_M, \sigma^2_M)\) starting from \((s, x)\). Furthermore, by \cite[Corollary 10.1.2]{SV}, we have \(P_{(s, x)} = P^M_{(s, x)}\) on \(\cF_{\rho^s_M}\) for all \(M > 0\) and \((s, x) \in \bR_+ \times \bR\).
	
	Next, take \(T, \varepsilon > 0\) and fix a bounded Borel function \(\phi \colon \bR\to \bR\). W.l.o.g., we assume that \(| \phi | \leq 1\). Take \((t^n, x^n)_{n \in \mathbb{Z}_+} \in [0, T) \times \bR\) such that \((t^n, x^n) \to (t^0, x^0)\). For \(M > T\), by Lemma \ref{lem: moment bound abvanced}, there exists a constant \(C > 0\), which is independent of \(M\), such that, for all \(n \in \mathbb{Z}_+\),
	\[
	P_{(t^n, x^n)} ( \rho^{t^n}_M \leq T) = P_{(t^n, x^n)} \Big( \sup_{s \in [t^n, T]} |X_s| \geq M \Big) \leq \frac{C}{M}.
	\]
	We take \(M > T\) large enough such that 
	\[
	\sup_{n \in \mathbb{Z}_+} P_{(t^n, x^n)} ( \rho^{t^n}_M \leq T) \leq \varepsilon.
	\]
	Thanks to \cite[Exercise 7.3.5]{SV} (see also \cite[Theorem 7.1.9, Exercise 7.3.3]{SV}), there exists an \(N \in \mathbb{N}\), which in particular depends on \(M\), such that, for all \(n \geq N\),
	\[
	\big| E^{P^M_{(t^n, x^n)}} \big[ \phi (X_T) \big] - E^{P^M_{(t^0, x^0)}}\big[ \phi (X_T) \big] \big| \leq \varepsilon.
	\]
	Now, for all \(n \geq N\), we obtain 
	\begin{align*}
	\big| E^{P_{(t^n, x^n)}} \big[ \phi (X_T) \big] &- E^{P_{(t^0, x^0)}}\big[ \phi (X_T) \big] \big| 
	\\&\leq \big| E^{P^M_{(t^n, x^n)}} \big[ \phi (X_T) \1_{\{T < \rho^{t_n}_M\}} \big] - E^{P^M_{(t^0, x^0)}}\big[ \phi (X_T) \1_{\{T < \rho^{t^0}_M\}} \big] \big| 
	\\&\hspace{4.925cm} + P_{(t^n, x^n)} ( \rho^{t^n}_M \leq T) +  P_{ (t^0, x^0)} (\rho^{t^0}_M \leq T)
	\\&\leq \big| E^{P^M_{(t^n, x^n)}} \big[ \phi (X_T) \1_{\{T < \rho^{t_n}_M\}} \big] - E^{P^M_{(t^0, x^0)}}\big[ \phi (X_T) \1_{\{T < \rho^{t^0}_M\}} \big] \big|  +2 \varepsilon
	\\&\leq \big| E^{P^M_{(t^n, x^n)}} \big[ \phi (X_T) \big] - E^{P^M_{(t^0, x^0)}}\big[ \phi (X_T) \big] \big| 
		\\&\hspace{4.925cm} + P^M_{(t^n, x^n)} ( \rho^{t^n}_M \leq T) +  P^M_{ (t^0, x^0)} (\rho^{t^0}_M \leq T)  +2 \varepsilon
		\\&\leq 3 \varepsilon + P_{(t^n, x^n)} ( \rho^{t^n}_M \leq T) +  P_{ (t^0, x^0)} (\rho^{t^0}_M \leq T) 
		\leq 5 \varepsilon,
	\end{align*}
	where we use that \(\mathcal{F}_T \cap \{T < \rho^s_M\} \subset \cF_{\rho^s_M}, \{\rho^s_M \leq T\} \in \cF_{\rho^s_M}\) and that \(P_{(s, x)} = P^M_{(s, x)}\) on \(\cF_{\rho^s_M}\).
	This proves the strong Feller property of \(\p=\{P_{(s, x)} \colon (s, x) \in \bR_+ \times \bR\}\). 
	
	Take \(0 \leq s < T\) and let \(\phi \colon \bR \to \bR\) be a continuous function vanishing at infinity such that \(|\phi| \leq 1\). By the strong Feller property, the map \(x \mapsto E^{P_{(s, x)}} [ \phi (X_T) ]\) is continuous. We now adapt the proof of \cite[Theorem~1]{criens20SPA} to conclude the \(C_0\)--Feller property. Fix \(\varepsilon > 0\). As \(\phi\) vanishes at  infinity, there exists an \(M = M (\varepsilon) > 0\) such that \(|f (y)| \leq \varepsilon\) for all \(|y| > M\). We obtain 
	\begin{align} \label{eq: ineq to show C0}
	E^{P_{(s, x)}} \big[ \phi (X_T) \big] &\leq \varepsilon + P_{(s, x)} ( |X_T| \leq M ).
	\end{align}
	In the following we establish an estimate for the second term. Set \(V (y) := 1 / (1 + y^2)\) for \(y \in \bR\) and let \((b^{P_{(s, x)}}_{\cdot + s}, a^{P_{(s, x)}}_{\cdot + s})\) be the Lebesgue densities of the \(P_{(s, x)}\)-characteristics of the shifted coordinate process \(X_{\cdot + s}\).
	Then, by Condition \ref{cond: LG}, we get \((\llambda \otimes P_{(s, y)})\)-a.e.
	\begin{align*}
	b^{P_{(s,x)}}_{\cdot +  s}  V' (X_{\cdot + s}) + \frac{a^{P_{(s,  x)}}_{\cdot + s} V'' (X_{\cdot + s})}{2}  
	&= \frac{- 2 b^{P_{(s, x)}}_{\cdot +  s} X_{\cdot + s}}{(1 + X_{\cdot + s}^2)^2} + \frac{a^{P_{(s,  x)}}_{\cdot + s}}{2} \Big( \frac{8 X_{\cdot + s}^2}{(1 + X_{\cdot + s}^2)^3} - \frac{2}{(1 + X_{\cdot + s}^2)^2}\Big)
	\\&\leq \C \Big( \frac{|b^{P_{(s, x)}}_{\cdot +  s}| |X_{\cdot + s}|}{(1 + X_{\cdot + s}^2)^2} + \frac{a^{P_{(s,  x)}}_{\cdot + s}}{(1 + X_{\cdot + s}^2)^2}\Big)
	\\&\leq \C \Big( \frac{ |X_{\cdot + s}| + X_{\cdot + s}^2}{(1 + X_{\cdot + s}^2)^2} + \frac{1}{1 + X_{\cdot + s}^2} \Big)
	\\&\leq \frac{\C}{1 + X_{\cdot + s}^2} = \C V(X),
	\end{align*}
	where the constant \(\C > 0\) depends only on the linear growth constant from Condition \ref{cond: LG}. In the above computation, \(\C\) might have changed from line to line. In the following, let \(\C  > 0\) be the constant from the last inequality. By It\^o's formula, the process
	\[
	e^{- \C \cdot} V (X_{\cdot + s}) - \int_0^\cdot e^{- \C r} \Big( - \C V (X_{r + s}) + b^{P_{(s, x)}}_{r +  s}  V' (X_{r + s}) + \frac{a^{P_{(s,  x)}}_{r + s} V'' (X_{r + s})}{2} \Big) dr
	\]
	is a local \(P_{(s, x)}\)-martingale. Thus, as
	\[
	\int_0^\cdot e^{- \C r} \Big( - \C V (X_{r + s}) + b^{P_{(s, x)}}_{r +  s}  V' (X_{r + s}) + \frac{a^{P_{(s,  x)}}_{r + s} V'' (X_{r + s})}{2} \Big) dr
	\]
	is a decreasing process, \(e^{- \C \hspace{0.025cm} \cdot} V (X_{\cdot + s})\) is a local \(P_{(s, x)}\)-supermartingale and hence, as it is bounded, a true \(P_{(s, x)}\)-supermartingale. By Chebyshev's inequality and the supermartingale property, we obtain 
	\begin{align*}
	P_{(s, x)} ( |X_T| \leq M ) &= P_{(s, x)} ( V (X_T) \geq V (M) ) 
	\\&\leq \frac{ E^{P_{(s, x)}} \big[ V (X_T) \big] }{ V (M) }
	\\&\leq \frac{e^{\C T} V (x)}{V (M)}
	\longrightarrow 0 \text{ as } |x| \to \infty.
	\end{align*}
Using this observation and \eqref{eq: ineq to show C0}, we obtain the existence of a compact set \(K = K(\varepsilon) \subset \bR\) such that
\[
E^{P_{(s, x )}}\big[ \phi (X_T) \big] \leq 2 \varepsilon
\]
for all \(x \not \in K\). We conclude that \(x \mapsto E^{P_{(s, x)}} [ \phi (X_T) ]\) vanishes at infinity and therefore, that \(\p=\{P_{(s, x)} \colon (s, x) \in \bR_+ \times \bR\}\) has the \(C_0\)--Feller property.

Finally, we presume that the Conditions~\ref{cond: bdd} and \ref{cond: uniform ellipticity} hold aditionally. Then, \(\mu\) and \(\sigma^2\) are both bounded and \(\sigma^2\) is uniformly bounded away from zero. Thanks to these observations, it follows from \cite[Exercise~7.3.5]{SV} (see also \cite[Theorem~7.1.9, Exercises~7.3.3]{SV}), that \(\p\) is a uniform strong Feller family.
The proof is complete.
\end{proof}

\begin{proof}[Proof of Theorem \ref{theo: Feller selection}]
The theorem follows directly from the Theorems \ref{theo: strong Markov selection} and \ref{theo: selection is Feller}.
\end{proof}

\section{Uniform Feller Selection Principles: Proof of Theorems \ref{theo: UFSP} and \ref{theo: UFSP 2}} \label{sec: uniform selection}
\subsection{Proof of Theorem \ref{theo: UFSP}}
The following lemma can be seen as a version of \cite[Theorem~3]{hajek85} where a global Lipschitz assumption is replaced by a local H\"older and ellipticity assumption. The idea of proof is the same, i.e., we use a time change argument and the optional sampling theorem. 
\begin{lemma} \label{lem: convex order}
	Suppose that \(P\in \mathfrak{P}(\Omega)\) is such that \(X\) is a continuous local \(P\)-martingale starting at \(x_0\) with quadratic variation process \(\int_0^\cdot a_s^P ds\) such that \((\llambda \otimes P)\)-a.e. \(a^P > 0\). Let \(a \colon \mathbb{R} \to (0, \infty)\) be such that \(\sqrt{a}\) is locally H\"older continuous with exponent \(1/2\) and \(a (x) \leq \C(1 + |x|^2)\) for all \(x \in \mathbb{R}\). Suppose that \(P\)-a.s. \(a^P_t \leq a(X_t)\) for \(\llambda\)-a.a. \(t \in \mathbb{R}_+\).
	Then, the SDE 
	\begin{align}\label{eq: SDE comparison}
	d Y_t = \sqrt{a} (Y_t) d W_t, \quad Y_0 = x_0, 
	\end{align}
	satisfies strong existence and pathwise uniqueness and we denote its unique law by \(Q\).\footnote{Uniqueness in law follows from pathwise uniqueness by the Yamada--Watanabe theorem (\cite[Proposition 5.3.20]{KaraShre}).} Moreover, for every convex function \(\psi \colon \mathbb{R} \to \mathbb{R}\) of polynomial growth, i.e., such that 
	\[
	| \psi (x) | \leq \C(1 + |x|^m) \text{ for some } m \in \mathbb{N}, 
	\]
	and time \(T \in \bR_+\), we have 
	\[
	E^P \big[ \psi (X_T) \big] \leq E^Q \big[ \psi (X_T)\big].
	\]
\end{lemma}
\begin{proof}
	The fact that the SDE \eqref{eq: SDE comparison} satisfies strong existence and pathwise uniqueness is classical (see, e.g., \cite[Corollary 5.5.10, Remark 5.5.11]{KaraShre}).
	We now prove the second claim. 	Clearly, it suffices to consider \(T > 0\).
	Define 
	\[
	L_t := \begin{cases} \int_0^t \frac{a^P_s ds}{a(X_s)},& t \leq 2T, \\
	L_{2T} + (t - 2T),& t \geq 2T.\end{cases}
	\]
	By our assumptions, \(P\)-a.s. \(L\) is continuous, strictly increasing, finite and \(L_t \leq t\) for all \(t \in \mathbb{R}_+\). Furthermore, \(P\)-a.s. \(L_t \to \infty\) as \(t \to \infty\).
	Denote the right inverse of \(L\) by \(S\), i.e., define \[S_t := \inf \{s \geq 0 \colon L_s > t\}\] for \(t \in \mathbb{R}_+\). By the above properties of \(L\), it is well-known (\cite[p. 180]{RY}) that \(P\)-a.s. \(S\) is also continuous, strictly increasing and finite, and furthermore, \(S_t \geq t\) for all \(t \in \mathbb{R}_+\).
	Using standard rules for Stieltjes integrals (see, e.g., \cite[Proposition V.1.4]{RY}), we obtain that \(P\)-a.s. for all \(t \in [0, L_T]\)
	\[
	S_t = \int_0^{S_t} \frac{a (X_s)}{a^P_s} d L_s = \int_0^t \frac{a (X_{S_s})}{a^P_{S_s}} d L_{S_s} = \int_0^t \frac{a (X_{S_s})}{a^P_{S_s}} ds.
	\]
	In other words, \(P\)-a.s. 	
	\[
	\1_{[0, L_T]} (t) d S_t = \1_{[0, L_T]}(t) \frac{a (X_{S_t})}{a^P_{S_t}} dt.
	\]
	By \cite[Proposition V.1.5]{RY}, the time changed process \(X_S\) is a continuous local \(P\)-martingale (for a time changed filtration) such that, \(P\)-a.s. for all \(t \in [0, L_T]\), we have
	\[
	\langle X_S, X_S \rangle_{t} = \langle X, X\rangle_{S_{t}} = \int^{S_{t}}_0 a^P_s ds =  \int_0^t a^P_{S_s} d S_s = \int_0^t a (X_{S_s}) ds.
	\]
	Thus, it is classical (\cite[Proposition 5.4.6]{KaraShre}) that, possibly on a standard extension of the underlying probability space, there exists a one-dimensional Brownian motion \(W\) such that 
	\[
	X_{S_{\cdot \wedge L_T}} = x_0 + \int_0^{\cdot \wedge L_T} \sqrt{a}(X_{S_s}) dW_s.
	\]
	With little abuse of notation, we denote the underlying probability measure still by \(P\).
	Thanks to the strong existence property of the SDE \eqref{eq: SDE comparison}, there exists a continuous adapted process \(Y\) such that 
	\[
	Y = x_0 + \int_0^\cdot \sqrt{a} (Y_s) d W_s.
	\]
	As the SDE \eqref{eq: SDE comparison} satisfies pathwise uniqueness, it follows from \cite[Lemma 3]{criens20SPA} that \(P\)-a.s. \(X_{S_{ \cdot \wedge L_T}} = Y_{\cdot \wedge L_T}\).
	Notice that, by the linear growth assumption on \(\sqrt{a}\), the process \(Y\) is a \(P\)-martingale. Indeed, this follows readily from a second moment bound (see, e.g., \cite[Problem 5.3.15]{KaraShre}), which implies integrability of the quadratic variation process. 
	
	We are in the position to complete the proof. Let \(\psi\) be a convex function of polynomial growth. Using again the linear growth assumption, we have polynomial moment bounds (see, e.g., \cite[Problem 5.3.15]{KaraShre}) which imply that \(\psi (Y_t) \in L^1(P)\) for all \(t \in \mathbb{R}_+\). Consequently, \(\psi (Y)\) is a \(P\)-submartingale. As \(P\)-a.s. \(L_T \leq T\), using the optional sampling theorem, we finally obtain that 
	\[
	E^P \big[ \psi (X_T) \big] = E^P \big[ \psi (X_{S_{L_T}})\big] = E^P \big[ \psi (Y_{L_T}) \big] 
	\leq E^P \big[ \psi (Y_T)\big] = E^Q \big[ \psi (X_T)\big].
	\] 
	The proof is complete.
\end{proof}

\begin{proof}[Proof of Theorem \ref{theo: UFSP}]
We set 
\[
a^* (x) := \sup \big\{ a (f, x) \colon f \in F \big\}
\]
for \(x \in \mathbb{R}\). 
Notice that \(a^*\) is locally H\"older continuous with exponent \(1/2\) thanks to Condition \ref{cond: local holder}. Furthermore, as \(a > 0\) by Condition \ref{cond: ellipticity}, compactness of \(F\) and continuity of \(a\) in the control variable (Condition \ref{cond: continuity in control}) show that \(a^* > 0\). Finally, \(\sqrt{a^*}\) is of linear growth by Condition \ref{cond: LG}.
For every \(x\in \mathbb{R}\), let \(P^*_x\) be the unique law of a solution process to the SDE
\[
d Y_t = \sqrt{a^*} (Y_t) d W_t, \quad Y_0 = x, 
\]
where \(W\) is a one-dimensional Brownian motion. The existence of \(P^*_x\) is classical (or follows from Lemma~\ref{lem: convex order}).
Moreover, \cite[Corollary 10.1.4, Theorem 10.2.2]{SV} yield that \(\{P^*_x \colon x \in \mathbb{R}\}\) is a strong Feller family. It is left to prove the formula \eqref{eq: USFSP}. 
Take \(\psi \in \mathbb{G}_{cx}\) and \(T \in \bR_+\). 
Now, Lemma~\ref{lem: convex order} yields that 
\[
\mathcal{E}^x (\psi (X_T)) \leq E^{P^*_x} \big[ \psi (X_T) \big].
\]
As \(P^*_x \in \cC(0, x)\), we also have 
\[
E^{P^*_x} \big[ \psi (X_T) \big] \leq \mathcal{E}^x (\psi (X_T)).
\]
Putting these pieces together yields the formula \eqref{eq: USFSP} and hence, the proof is complete.
\end{proof}

\subsection{Proof of Theorem \ref{theo: UFSP 2}}
We set 
\[
b^* (x) := \sup \big\{ b (f, x) \colon f \in F \big\}
\]
for \(x \in \mathbb{R}\). As \(F\) is compact and \(b\) is continuous, \(b^*\) is continuous by Berge's maximum theorem (\cite[Theorem 17.31]{hitchi}). 
Moreover, \(b^*\) and \(a^*\) are of linear growth by Condition \ref{cond: LG}.
Consequently, taking Condition \ref{cond: holder} into consideration, \cite[Theorem 4.53]{engelbert1991strong} and \cite[Theorem 10.2.2]{SV} yield that the SDE
\begin{align} \label{eq: SDE comparison 2}
d Y_t = b^* (Y_t) dt + \sqrt{a^*} (Y_t) d W_t
\end{align}
satisfies weak existence and pathwise uniqueness. Let \(P^*_x\) be the unique law of a solution process starting at \(x \in \mathbb{R}\). Then, by \cite[Corollary 10.1.4, Theorem 10.2.2]{SV}, the family \(\{P^*_x \colon x \in \mathbb{R}\}\) is strongly Feller. Let \(\psi \colon \Omega \to \mathbb{R}\) be a bounded increasing Borel function. By construction, we have
\[
E^{P^*_x} \big[ \psi  \big] \leq \mathcal{E}^x (\psi), \quad x \in \mathbb{R}.
\]
Take \(x \in \mathbb{R}\) and \(P \in \cR( x )\). Thanks to  \cite[Theorem VI.1.1]{ikeda2014stochastic}, possibly on a standard extension of \((\Omega, \mathcal{F}, \F, P)\), there exists a solution process \(Y\) to the SDE \eqref{eq: SDE comparison 2} with \(Y_0 = x\) such that a.s. \(X_t \leq Y_t\) for all \(t \in \mathbb{R}_+\). Clearly, this yields that 
\[
E^P \big[ \psi \big] \leq E^{P^*_x} \big[ \psi \big],
\]
and taking the \(\sup\) over all \(P \in \cR(x)\) finally gives 
\[
\mathcal{E}^x (\psi) \leq  E^{P^*_x} \big[ \psi \big].
\]
The proof is complete. \qed

\section{A nonlinear Kolmogorov Equation: Proof of Theorems \ref{thm: new viscosity no unique} and \ref{theo: viscosity with ell}} \label{sec: viscosity property}
\subsection{Proof of Theorem \ref{thm: new viscosity no unique}}
It follows from \cite[Theorem 4.3]{CN22} that \( v \) is a weak sense viscosity solution to \eqref{eq: PDE}.
Hence, it suffices to show continuity of
\([0, T] \times \mathbb{R} \ni (t,x) \mapsto \cE^x(\psi(X_{T-t})) \).
Due to Theorem~\ref{thm: new very main}, the map \(x \mapsto \cE^x (\psi (X_{T- t}))\) is continuous for every \(t \in [0, T]\). 
We now show that, for any compact set~\(K \subset \mathbb{R}\),
\[
\sup_{x \in K} \big| \cE^x (\psi (X_{T - t})) - \cE^x (\psi (X_{T - s}))\big| \to 0 \text{ as } s \to t.
\]
Clearly, by the triangle inequality, this then implies continuity of \([0, T] \times \mathbb{R} \ni (t,x) \mapsto \cE^x(\psi(X_{T-t})) \).
Take a compact set \(K \subset \mathbb{R}\), \(r, \varepsilon > 0\) and \(s, t \in [0, T]\).
In the following \(\C > 0\) denotes a generic constant which is independent of \(r, \varepsilon, s, t\).
By Lemma \ref{lem: maximal inequality}, we get
\begin{align*}
\sup_{x \in K} \sup_{P \in \cR(x)} P ( |X_{T - t}| > r \text{ or } |X_{T - s}| > r) \leq \tfrac{\C}{r}.
\end{align*}
Using Lemma \ref{lem: maximal inequality} again, we further obtain that
\begin{align*}
   \sup_{x \in K} \big| \cE^x & (\psi (X_{T - t})) - \cE^x (\psi (X_{T - s}))\big| 
   \\ & \leq \sup_{x \in K} \sup_{P \in \cR(x)} E^P \big[ | \psi (X_{T - t}) - \psi (X_{T - s})| \big]
   \\
   &\leq \sup_{x \in K} \sup_{P \in \cR(x)} E^P \big[ | \psi (X_{T - t}) - \psi (X_{T - s})| \1_{\{ |X_{T - t} - X_{T - s}| \leq \varepsilon\} \hspace{0.05cm}\cap \hspace{0.05cm} \{ |X_{T - t}| > r \text{ or } |X_{T - s}| > r \}}\big] 
   \\&\hspace{2cm} + \sup_{x \in K} \sup_{P \in \cR(x)} E^P \big[ | \psi (X_{T - t}) - \psi (X_{T - s})| \1_{\{ |X_{T - t} - X_{T - s}| \leq \varepsilon,\ |X_{T - t}| \leq r,\ |X_{T - s}| \leq r \}}\big]
   \\&\hspace{2cm} +\sup_{x \in K} \sup_{P \in \cR(x)} E^P \big[ | \psi (X_{T - t}) - \psi (X_{T - s})| \1_{\{ |X_{T - t} - X_{T - s}| > \varepsilon \}}\big]
   \\&\leq \frac{2 \C \|\psi\|_\infty}{r} + \sup \big\{ |\psi (z) - \psi(y)| \colon |z-y| \leq \varepsilon, |z| \leq r, |y| \leq r \big\} \\&\hspace{2cm}+ 2 \|\psi\|_\infty \sup_{x \in K} \sup_{P \in \cR(x)} P(|X_{T - t} - X_{T - s}| > \varepsilon)
   \\&\leq \frac{2 \C \|\psi\|_\infty}{r} + \sup \big\{ |\psi (z) - \psi(y)| \colon |z-y| \leq \varepsilon, |z| \leq r, |y| \leq r \big\} + \frac{ \C \|\psi\|_\infty}{\varepsilon} |t - s|^{1/2}.
   \end{align*}
   Notice that the middle term converges to zero as \(\varepsilon \to 0\), since continuous functions are uniformly continuous on compact sets. 
   Thus, choosing first \(r\) large enough and then \(\varepsilon\) small enough, we can make 
   \[
   \sup_{x \in K} \big| \cE^x (\psi (X_{T - t})) - \cE^x (\psi (X_{T - s}))\big| 
   \]
   arbitrarily small when \(s \to t\). This yields the claim.
\qed

\subsection{Proof of Lemma~\ref{lem: en upper lower}} \label{sec: pf en upper lower}
We only detail the proof for the subsolution property of \(v^*\). The supersolution property of \(v_*\) will follow in the same spirit.
Let \(\phi \in  C^{2, 3}_b([0, T] \times \bR; \bR)\) such that \(\phi \geq v^*\) and \(\phi (t, x) = v^*(t, x)\) for some \((t, x) \in [0, T) \times \bR \). Notice that we use a test function of higher regularity than in our definition of ``viscosity subsolution''. This is without loss of generality, cf. \cite[Lemma~2.4, Remark~2.5]{hol16}.
There exists a sequence \((t^n, x^n)_{n = 1}^\infty \subset [0, T) \times \bR\) such that \((t^n, x^n) \to (t, x)\) and 
\[
v^* (t, x) = \lim_{n \to \infty} v (t^n, x^n). 
\]
We take an arbitrary \(u \in (0, T - t)\).
By the dynamic programming principle (\cite[Theorem 3.1]{CN22}), for every \(n \in \mathbb{N}\), we obtain that 
\begin{equation} \label{eq: main ev upper}
	\begin{split}
		0 &= \sup_{P \in \cR (x^n)} E^P \big[ v (t^n + u, X_u) \big] - v (t^n, x^n)
		\\&\leq \sup_{P \in \cR (x^n)} E^P \big[ v^* (t^n + u, X_u) \big] - v^* (t, x) + v^* (t, x) - v (t^n, x^n)
  \\&= \sup_{P \in \cR (x^n)} E^P \big[ v^* (t^n + u, X_u) \big] - \phi (t, x) + v^* (t, x) - v (t^n, x^n)
		\\&\leq \sup_{P \in \cR (x^n)} E^P \big[ \phi (t^n + u, X_u) \big] - \phi (t^n, x^n) + \phi (t^n, x^n) - \phi (t, x) + v^* (t, x) - v (t^n, x^n).
	\end{split}
\end{equation}
Now, we use an argument as in the proof of \cite[Proposition~5.4]{neufeld2017nonlinear}. 
Take \(P \in \cR(x^n)\). It\^o's formula yields that \(P\)-a.s.
\begin{align*}
	\phi (t^n + u, X_u) - \phi (t^n, x^n) = \int_0^u \big[ \partial_t \phi (t^n + s, X_s) + b^P_s \partial_x \phi (t^n + s, X_s) + \tfrac{1}{2} c^P_s & \partial^2_x \phi (t^n + s, X_s) \big] ds 
	\\&+ \text{local \(P\)-martingale}. 
\end{align*}
Thanks to the linear growth Condition~\ref{cond: LG} and Lemma~\ref{lem: maximal inequality}, it follows that the local \(P\)-martingale part is actually a true \(P\)-martingale. Hence, we obtain that 
\begin{align*}
	E^P \big[ \phi (t^n &+ u, X_u) - \phi (t^n, x^n) \big] 
 \\&= \int_0^u E^P \big[ \partial_t \phi (t^n + s, X_s) + b^P_s \partial_x \phi (t^n + s, X_s) + \tfrac{1}{2} c^P_s \partial^2_x \phi (t^n + s, X_s) \big] ds
	\\&=: I_u.
\end{align*}
As \(\phi \in C^{2, 3}_b ([0, T] \times \bR; \bR)\), the derivatives \(\partial_t \phi, \partial_x \phi\) and \(\partial^2_x \phi\) are (globally) Lipschitz continuous. Hence, we obtain that 
\begin{align*}
	\int_0^u \partial_t \phi (t^n + s , X_s) ds &\leq u \partial_t \phi (t^n, x^n) + \int_0^u | \partial_t \phi (t^n + s , X_s) - \partial_t \phi (t^n, x^n) | ds 
	\\&\leq u \partial_t \phi (t^n, x^n) + \C  \int_0^u \big( s + |X_s - x^n| \big) ds 
	\\&= u \partial_t \phi (t^n, x^n) + \C u^2 + \C  \int_0^u|X_s - x^n| ds.
\end{align*}
By Lemma~\ref{lem: maximal inequality}, this implies that 
\begin{align*}
	E^P \Big[ 	\int_0^u \partial_t \phi (t^n + s , X_s) ds \Big] &\leq u \partial_t \phi (t^n, x^n) + \C u^2 + \C \int_0^u E^P \big[ |X_s - x^n| \big] ds
	\\&\leq u \partial_t \phi (t^n, x^n) + \C u^2 + \C \int_0^u s^{1/2} ds 
	\\&\leq u \partial_t \phi (t^n, x^n) + \C u^{3/2}.
\end{align*}
Similarly, using also the linear growth Condition~\ref{cond: LG}, we obtain that 
\begin{align*}
	E^P \Big[ \int_0^u \big| b^P_s \big| \, \big| \partial_x &\phi (t^n + s, X_s) - \partial_x \phi (t^n, x^n) \big| ds \Big] 
	\\&\leq E^P \Big[ \int_0^u \C \big( s + |X_s - x^n| \big) \big(1 + |X_s| \big) ds \Big] 
 \\&\leq E^P \Big[ \int_0^u \C \big( s + |X_s - x^n| \big) \big(1 + |X_s - x^n| \big) ds \Big] 
	\\&\leq \C u^{3/2}, 
\end{align*}
and that 
\[
E^P \Big[ \int_0^u \big| c^P_s \big| \, \big| \partial_x^2 \phi (t^n + s, X_s) - \partial_x^2 \phi (t^n, x^n) \big| ds \Big] \leq \C u^{3/2}. 
\]
In summary, we conclude that 
\begin{align*}
	I_u \leq \C u^{3/2} + u \partial_t \phi (t^n, x^n) + \int_0^u E^P \big[ G^n (X_s) \big] ds, 
\end{align*}
where 
\[
G^n (x) := \sup \Big\{ b (f, x) \partial_x \phi (t^n, x^n) + \tfrac{1}{2} a (f, x) \partial^2_x \phi (t^n, x^n) \colon f \in F \Big\}. 
\]
Using Condition \ref{cond: Lipschitz continuity}, we obtain that 
\begin{align*}
	G^n (x) \leq G (t^n, x^n, \phi) + C |x - x^n|.
\end{align*}
Hence, using again Lemma \ref{lem: maximal inequality}, it follows that 
\begin{align*}
	I_u \leq \C u^{3/2} + u \partial_t \phi (t^n, x^n) + u G (t^n, x^n, \phi) + \C u^{3/2}.
\end{align*}
Thanks to Condition~\ref{cond: continuity} and Berge's maximum theorem, the map \((s, y) \mapsto G (s, y, \phi)\) is continuous. 
Recalling \eqref{eq: main ev upper} and taking the limit \(n \to \infty\), we obtain that 
\begin{align*}
	0 &\leq \C u^{3/2} + u \partial_t \phi (t^n, x^n) + u G (t^n, x^n, \phi) + \phi (t^n, x^n) - \phi (t, x) + v^* (t, x) - v (t^n, x^n)
	\\&\to \C u^{3/2} + u \partial_t \phi (t, x) + u G (t, x, \phi).
\end{align*}
Dividing the last term by \(u\) and then letting \(u \searrow 0\), we finally conclude that 
\begin{align*}
	0 \leq \partial_t \phi (t, x) + G (t, x, \phi).
\end{align*}
We proved that \(v^*\) is a viscosity subsolution to \eqref{eq: PDE}. \qed

\section{Relation to the Nisio Semigroup} \label{sec: pf nisio}

\subsection{Proof of Lemma~\ref{lem: Sf semigroup}}
We first establish an auxiliary result. Recall that \(Y^{f, x}\) are continuous processes with dynamics
\[
d Y_t^{f, x} = b (f, Y^{f, x}_t) dt + \sqrt{a} (f, Y^{f, x}_t) d W_t, \quad Y^{f, x}_0 = x,
\]
which are defined on the same probability space w.r.t. the same Brownian motion \(W\).
\begin{lemma} \label{lem: A1}
Assume that Condition~\ref{cond: fix f cond} holds and let \(\beta\) be the uniform (in the \(F\)-variable) Lipschitz constant of the drift coefficient. Then, for every \(f \in F, t \in \bR_+\) and \(x, y \in \bR\), 
\[
E^P \big[ | Y^{f, x}_t - Y^{f, y}_t | \big] \leq |x - y| \, e^{\beta t}.
\]
\end{lemma}
\begin{proof}
    Take \(f \in F, x, y \in \bR\) and set
    \[
    Z := Y^{f, x} - Y^{f, y} = x - y + \int_0^\cdot \big( b (f, Y^{f, x}_t) - b (f, Y^{f, y}_t) \big) dt + \int_0^\cdot \big( \sqrt{a} (f, Y^{f, x}_t) - \sqrt{a} (f, Y^{f, y}_t) \big) dW_t.
    \]
    Fix \(M > 0\), set 
    \[
    T_M := \inf \big \{t \geq 0 \colon |Y^{f, x}_t| \vee |Y^{f, y}_t| \geq M \big \}.
    \]
    For every \(t > 0\), we have \(P\)-a.s.
    \begin{align*}
        \int_0^{t \wedge T_M} \frac{\1_{\{ 0 < Z_s \leq \varepsilon\}} d [ Z, Z ]_s}{|Z_s|} &\leq \int_0^{t \wedge T_M} \frac{( \sqrt{a} (f, Y^{f, x}_s) - \sqrt{a} (f, Y^{f, y}_s)  )^2 ds }{|Z_s|}
        \leq \C t < \infty.
    \end{align*}
    Hence, \cite[Lemma~IX.3.3]{RY} yields that \(P\)-a.s. \(L^0_{\cdot \wedge T_M} (Z) = 0\), where \(L^0(Z)\) denotes the semimartingale local time of \(Z\) in zero. As \(P\)-a.s. \(T_M \nearrow \infty\) with \(M \to \infty\), this implies that \(P\)-a.s. \(L^0 (Z) = 0\).
    Now, we get from Tanaka's formula (\cite[Theorem~VI.1.2]{RY}) that \(P\)-a.s.
    \[
    |Z| = |x - y| + \int_0^\cdot \on{sgn} (Z_s) d Z_s.
    \]
    Using that the local martingale part of the stochastic integral is a true martingale (which follows from the linear growth part of Condition~\ref{cond: fix f cond}), we obtain, for every \(t \in \bR_+\), that 
    \begin{align*}
        E^P \big[ |Z_t| \big] &\leq |x - y| + \int_0^t E^P \big[ |b (f, Y^{f, x}_s) - b (f, Y^{f, y}_s) | \big] ds 
        \\&\leq |x - y| + \int_0^t \beta E^P \big[ |Z_s| \big] ds.
    \end{align*}
    Gronwall's lemma yields that \(E^P [ |Z_t|] \leq |x - y| e^{\beta t}\). This is the claim.
\end{proof}

\begin{proof}[Proof of Lemma~\ref{lem: Sf semigroup}]
As we already know that \((S^f_t)_{t \in \bR_+}\) has the semigroup property, it suffices to prove that \(S^f_t\) maps \(\uc_b (\bR; \bR)\) into \(\uc_b (\bR; \bR)\). Take \(u \in \uc_b (\bR; \bR)\) and \(\varepsilon > 0\). By the definition of uniform continuity, there exists a constant \(\delta = \delta (\varepsilon) > 0\) such that 
\[
x, y \in \bR,\, |x - y| < \delta \ \Longrightarrow \ | u (x) - u (y) | < \varepsilon / 2. 
\]
For all \(x, y \in \bR\) and \(t > 0\), Lemma~\ref{lem: A1} yields that 
\begin{align*}
\big| S^{f}_t (u) (x) - S^f_t (u) (y) \big| &\leq E \big[ \big| u (Y^{f, x}_t) - u (Y^{f, y}_t) \big| \big] 
\\&\leq \varepsilon / 2 + 2 \|u\|_\infty P (| Y^{f, x}_t - Y^{f, y}_t | \geq \delta) 
\\&\leq \varepsilon / 2 + |x - y| \, 2 \|u\|_\infty e^{\beta t}/ \delta.
\end{align*}
Therefore, there exists a \(\delta' = \delta' (\varepsilon) > 0\) such that 
\[
x, y \in \bR,\, |x - y| < \delta' \ \Longrightarrow \ \big| S^{f}_t (u) (x) - S^f_t (u) (y) \big| < \varepsilon.
\]
This means that \(S^f_t (u) \in \uc_b (\bR; \bR)\) and hence, completes the proof.
\end{proof}

\subsection{Proof of Proposition~\ref{prop: first NR}}
    (i). We have to check the assumptions (A1) and (A2) from \cite{NR}. Then, the claim follows from \cite[Theorem~2.5]{NR}. 
    In our setting, Condition~\ref{cond: fix f cond} implies these standing assumptions by virtue of the Lemmata~\ref{lem: Sf semigroup} and \ref{lem: A1}.

    \smallskip
    (ii). We denote the space of all bounded twice differentiable uniformly continuous functions from \(\bR\) into \(\bR\) with bounded and uniformly continuous derivatives by \(\uc_b^{\,2} (\bR; \bR)\). Take a function \(u \in \uc_b^{\, 2} (\bR; \bR)\) and set 
    \[
    L_u := \sup \big\{ \big|b (f, x) u' (x) + \tfrac{1}{2} a ( f, x ) u'' (x)\big| \colon f \in F, x \in \bR \big\}.
    \]
    Thanks to Condition~\ref{cond: bdd}, \(L_u < \infty\). Take \(x \in \bR\). By It\^o's formula, we obtain that 
    \[
    \big| E^P \big[ u (Y^{f, x}_t) \big] - u (x) \big| \leq t L_u
    \]
    for all \(t \in \bR_+\).
    Hence, as \(\uc^{\, 2}_b (\bR; \bR)\) is dense in \(\uc_b(\bR; \bR)\) for the uniform topology, it follows from \cite[Proposition~3.5]{NR} that \((\mathscr{S}_t)_{t \in \bR_+}\) is strongly continuous. The proof is complete. 

    \smallskip 
    (iii). 
    Let \(C^2_0 (\bR; \bR)\) be the space of all twice continuously differentiable functions \(g \colon \bR \to \bR\) such that \(g, g', g'' \in C_0 (\bR; \bR)\). 
    We set 
        \[
    D (A^f) := \Big\{ u \in \uc_b (\bR; \bR) \colon \exists g \in \uc_b (\bR; \bR) \text{ s.t. } \lim_{h \searrow 0} \Big\| \frac{S^f_h (u) - u}{h} - g \Big\|_\infty = 0 \Big\}.
    \]
    It is well-known that each \((S^f_t)_{t \in \bR_+}\) is a strongly continuous semigroup on \(C_0(\bR; \bR)\), cf. \cite[Theorem~32.11]{Kallenberg}.
    Hence, \(C^2_0 (\bR; \bR) \subset \bigcap_{f \in F} D (A^f)\) by \cite[Proposition~VII.1.7]{RY}.
    Now, thanks to \cite[Proposition~5.4]{NR}, the claim follows once we prove that for every \(y \in \bR\) and \(\delta > 0\) there exists a function \(\phi = \phi_{y, \delta} \in C_0^2 (\bR; \bR)\) such that \(\phi (y) = 1, 0 \leq \phi \leq 1\) and \(\sup_{f \in F} \|A^f (\phi) \|_\infty \leq \delta\).
    Fix \(y \in \bR\) and \(\delta > 0\). For \(R > 0\), set 
    \[
    \phi (x) := \frac{R}{R + (x - y)^2}.
    \]
    Clearly, \(\phi (y) = 1\) and \(\phi \in C_0 (\bR; [0, 1])\). Furthermore, as
    \begin{align*}
        \phi' (x) &= \frac{- 2 R (x - y)}{(R + (x - y)^2)^2}, \\ 
        \phi'' (x) &= \frac{8 R (x - y)^2}{(R + (x - y)^2)^3} - \frac{2R}{(R + (x - y)^2)^2}, 
    \end{align*}
    we have \(\phi \in C^2_0 (\bR; \bR)\). Using Condition~\ref{cond: bdd}, we obtain that 
    \begin{align*}
        | A^f (\phi) (x) | &\leq \C \Big[ \frac{R |x - y|}{(R + (x - y)^2)^2} + \frac{8 R (x - y)^2}{(R + (x - y)^2)^3} + \frac{2}{R} \Big]
        \\&\leq \C \Big[ \frac{32}{27 R} + \frac{3 \sqrt{3}}{16 \sqrt{R}} + \frac{2}{R} \Big], 
    \end{align*}
    where we use that \(z \mapsto R |z|/(R + z^2)^2\) attains its maximum at \(z^2 = R/3\) and \(z \mapsto 8 R z^2/ (R + z^2)^3\) attains its maximum at \(z^2 = R/2\). Now, we can choose \(R = R(\delta)\) large enough such that \(\sup_{f \in F} \|A^f ( \phi) \|_\infty \leq \delta\). The proof is complete.
    \qed

\subsection{Proof of Proposition~\ref{prop: extension}}
Thanks to Proposition \ref{prop: first NR}, it follows from \cite[Remark~5.4 (b, c)]{denk2018} that, for every \(t \in \bR_+\), there exists a unique sublinear operator \(\widehat{\mathscr{S}}_t \colon C_b (\bR; \bR) \to C_b (\bR; \bR)\) such that \(\mathscr{S}_t ( \phi) = \widehat{\mathscr{S}}_t (\phi)\) for all \(\phi \in \uc_b(\bR; \bR)\) that is continuous from above on \(C_b(\bR;\bR)\); see also \cite[Remark~5.3~(c)]{NR}. Moreover, \cite[Remark~5.4 (d)]{denk2018} ensures the semigroup property of 
\((\widehat{\mathscr{S}}_t)_{t \in \bR_+}\). This completes the proof.\qed

\subsection{Proof of Theorem~\ref{thm: nisio}}
    {\em Step 1.} Let \(\lip^2_b (\bR; \bR)\) be the space of all bounded Lipschitz continuous and twice continuously differentiable functions from \(\bR\) into \(\bR\) with bounded and Lipschitz continuous first and second derivatives.
    First, we prove that \(\lip^2_b(\bR; \bR) \subset \bigcap_{f \in F} D (A^f)\). Take \(u \in \lip^2_b (\bR; \bR)\).
    Recall the following fact: If \(g_1\) and \(g_2\) are bounded and Lipschitz continuous with Lipschitz constants \(L_1\) and \(L_2\), then \(g_1 g_2\) is also Lipschitz continuous with Lipschitz constant \(\|g_1\|_\infty L_2 + \|g_2\|_\infty L_1\). Using this fact and the Conditions~\ref{cond: bdd} and \ref{cond: Lipschitz continuity}, we obtain the existence of a constant \(\C = \C_u\), that is in particular independent of \(f\), such that, for all \(x, y \in \bR\),
    \begin{align} \label{eq: A lip}
        | A^f (u)(x) - A^f (u) (y) | \leq \C |x - y|.
    \end{align}
    Now, for \(h \in (0, 1)\) and \(x \in \bR\), It\^o's formula, the Lipschitz bound \eqref{eq: A lip}, the Burkholder--Davis--Gundy inequality and Condition~\ref{cond: bdd} yield that
    \begin{align*}
        \Big|\frac{S^f_h (u) (x) - u(x)}{h} - A^f (u) (x)\Big| &= \Big| \frac{1}{h} E^P \Big[ \int_0^h (A^f (u) (Y^{f, x}_s) - A^f (u) (x)) ds \Big]\Big| 
        \\&\leq \frac{\C}{h} \int_0^h E^P \big[ | Y^{f, x}_s - x |^2 \big]^{1/2} ds 
        \\&\leq \frac{\C}{h} \int_0^h s^{1/2} ds = \C h^{1/2}.
    \end{align*}
    This proves that \(u \in \bigcap_{f \in F} D(A^f)\).
    
    \smallskip
    {\em Step 2.}
    Next, we prove that 
    \begin{align} \label{eq: prop 4.2 NR}
    \sup_{f \in F} \| S^f_h (A^f (u)) - A^f (u) \|_\infty \to 0, \quad h \searrow 0,
    \end{align}
    for all \(u \in \lip^2_b (\bR; \bR)\). 
    For every \(h \in (0, 1)\), using \eqref{eq: A lip}, the boundedness of \(b\) and \(a\) and the Burkholder--Davis--Gundy inequality, we obtain that 
    \begin{align*}
        \big| E^P \big[ A^f (u) (Y^{f, x}_h) \big] - A^f (u) (x) \big| &\leq \C E^P \big[ | Y^{f, x}_h - x| \big] 
        \leq \C h^{1/2}.
    \end{align*}
    We stress that the constant \(\C\) does not depend on \(f, x\) and \(h\). Hence, \eqref{eq: prop 4.2 NR} holds. 

    \smallskip
    {\em Step 3.} Recalling that \((\mathscr{S}_t)_{t\in \bR_+}\) is strongly continuous by Proposition~\ref{prop: first NR}, and thanks to the Steps~1 and 2, and the boundedness Condition~\ref{cond: bdd}, it follows from \cite[Proposition~4.2]{NR} that the class \(\lip^{2}_b (\bR; \bR)\) is contained in the domain \(\mathcal{D}\) of the generator of \((\mathscr{S}_t)_{t\in \bR_+}\) that is defined on p. 4414 in \cite{NR}. Using this observation and \cite[Remark~4.4, Theorem~4.5]{NR}, we get, for every \(u_0 \in \uc_b (\bR; \bR)\) and every \(T > 0\), that the map \([0, T] \ni t \mapsto \mathscr{S}_{T - t} (u_0)\) is a viscosity solution to the PDE
    \[
 \partial_t u + \sup_{f \in F} A^f (u) = 0 \text{ on } [0, T) \times \bR, \quad u (T, \cdot\,) = u_0,
    \]
    where the class of test functions is now \(\lip^{1, 2}_b ( [0, T) \times \bR )\). It is well-known (see, e.g., \cite[Lemma~2.4,~Remark 2.5]{hol16} or \cite[p. 4419]{NR}) that the class \(\lip^{1, 2}_b ( [0, T) \times \bR )\) leads to the same definition of viscosity solution as the standard class \(C^{1, 2}_b ([0, T) \times \bR)\). Hence, it follows from Theorem~\ref{theo: viscosity with ell} that \(\mathscr{S}_{T - t} (u_0) = T_{T - t} (u_0)\) for all \(t \in [0, T]\). As \(u_0 \in \uc_b (\bR; \bR)\) and \(T > 0\) were arbitrary, we conclude that \((\mathscr{S}_t)_{t\in \bR_+} = (T_t)_{t \in \bR_+}\) on \(\uc_b(\bR; \bR)\). The proof is complete.
\qed


\end{document}